\documentclass[11pt,final]{amsart}
\usepackage[toc,page]{appendix}
\usepackage[active]{srcltx}
\usepackage[colorlinks=true,urlcolor=blue,
citecolor=red,linkcolor=blue,linktocpage,pdfpagelabels,bookmarksnumbered,bookmarksopen]{hyperref}
\usepackage{marginnote,soul}
\usepackage[left=2.95cm,right=2.95cm,top=2.8cm,bottom=2.8cm]{geometry}

\numberwithin{equation}{section}
\usepackage{mathrsfs}
\usepackage{color}
\usepackage{amsmath,bbm}
\usepackage{amsfonts}
\usepackage{amssymb}
\usepackage{graphicx}
\usepackage{tikz,caption}
\providecommand{\U}[1]{\protect\rule{.1in}{.1in}}
\definecolor{linkcolor}{rgb}{0.00,0.50,0.00}
\providecommand{\U}[1]{\protect\rule{.1in}{.1in}}
\def\dd{\,{\rm d}}
\newcommand{\adm}{\rm{adm}}
\newcommand{\press}{\rm{press}}
\newcommand{\argmin}{\rm{argmin}}
\newcommand{\id}{\rm{id}}
\newcommand{\Lip}{\rm{Lip}}

\def\d{\delta}
\def\g{\gamma}

\def\vphi{\varphi}

\def\ro{\varrho}

\def\t{\tau}
\def\e{\varepsilon}

\def\Om{\Omega}

\newcommand{\cB}{{\mathcal B}}

\newcommand{\cE}{{\mathcal E}}
\newcommand{\cH}{{\mathcal H}}

\newcommand{\cL}{{\mathcal L}}
\newcommand{\cK}{{\mathcal K}}

\newcommand{\cP}{{\mathcal P}}

\newcommand{\pical}{{\mathcal P}}

\newcommand{\fM}{{\mathfrak M}}

\newcommand{\ds}{\displaystyle}
\newcommand{\weak}{\rightharpoonup}
\newcommand{\weaks}{\stackrel{*}{\rightharpoonup}}
\newcommand{\ov}{\overline}
\renewcommand{\div}{\nabla\cdot}

\renewcommand{\P}{\mathcal{P}}
\newcommand{\R}{\mathbb{R}}
\newcommand{\T}{\mathbb{T}}
\renewcommand{\O}{\Omega}
\renewcommand{\rm}{\mathrm}

\tikzset { domaine/.style 2 args={domain=#1:#2} }
\tikzset{
xmin/.store in=\xmin, xmin/.default=-3, xmin=-3,
xmax/.store in=\xmax, xmax/.default=3, xmax=3,
ymin/.store in=\ymin, ymin/.default=-3, ymin=-3,
ymax/.store in=\ymax, ymax/.default=3, ymax=3,
}

\newcommand {\axes} {
\draw[->] (\xmin,0) -- (\xmax,0);
\draw[->] (0,\ymin) -- (0,\ymax);
}
\newcommand {\fenetre}
{\clip (\xmin,\ymin) rectangle (\xmax,\ymax);}

\newcommand{\one}{\mathbbm{1}}

\title[Diffusive evolution under density constraints]{Advection-diffusion equations\\ with density constraints}
\author{Alp\'ar Rich\'ard M\'esz\'aros}
\address{Laboratoire de Math\'ematiques d'Orsay, Univ. Paris-Sud, CNRS, Universit\'e Paris-Saclay, 91405 Orsay cedex, France}
\address[Current address] {Department of Mathematics, University of California at Los Angeles,  520 Portola Plaza, 
 Los Angeles, 90095, California,  USA}
\email[A.R. M\'esz\'aros]{alpar@math.ucla.edu}

\author{Filippo Santambrogio}
\address{Laboratoire de Math\'ematiques d'Orsay, Univ. Paris-Sud, CNRS, Universit\'e Paris-Saclay, 91405 Orsay cedex, France}
\email[F. Santambrogio]{filippo.santambrogio@math.u-psud.fr}
\date{\today}

\begin{document}
\newtheorem{theorem}{Theorem}[section]
\newtheorem{lemma}[theorem]{Lemma}
\newtheorem{corollary}[theorem]{Corollary}

\theoremstyle{definition}
\newtheorem{definition}[theorem]{Definition}
\newtheorem{example}[theorem]{Example}
\newtheorem{xca}[theorem]{Exercise}
\theoremstyle{remark}
\newtheorem{remark}[theorem]{Remark}
\numberwithin{equation}{section}

\thispagestyle{plain}
\maketitle

\allowdisplaybreaks

\begin{abstract} In the spirit of the macroscopic crowd motion models with hard congestion (i.e. a strong density constraint $\rho\leq 1$) introduced by Maury {\it et al.} some years ago, we analyze a variant of the same models where diffusion of the agents is also taken into account. From the modeling point of view, this means that individuals try to follow a given spontaneous velocity, but are subject to a Brownian diffusion, and have to adapt to a density constraint which introduces a pressure term affecting the movement. From the PDE point of view, this corresponds to a modified Fokker-Planck equation, with an additional gradient of a pressure (only living in the saturated zone $\{\rho=1\}$) in the drift. The paper proves existence and some estimates, based on optimal transport techniques. 
\end{abstract}

\section{Introduction}
In the past few years modeling crowd behavior has become a very active field of applied mathematics.  Beyond their importance in real life applications, these modeling problems serve as basic ideas to understand many other phenomena coming for example from biology (cell migration, tumor growth, pattern formations in animal populations, etc.), particle physics and economics. A first non-exhaustive list of references for these problems is \cite{Cha1, Col, Cos, CriPicTos, Dog, Helb1, Helb3, Hug1, Hug2, crowd1}. A very natural question in all these models is the problem of congestion phenomenon: in many practical situations, very high quantities of individuals could try to occupy the same spot, which could be impossible, or lead to strong negative effects on the motion, because of natural limitations on the crowd density. 

These phenomena have been studied by using different models, which could be either ``microscopic'' (based on ODEs on the motion of a high number of agents) or ``macroscopic'' (describing the agents via their density and velocity, typically with Eulerian formalism). Let us concentrate on the macroscopic models, where the density $\rho$ plays a crucial role. These very same models can be characterized either by ``soft congestion'' effects (i.e. the higher the density the slower the motion), or by ``hard congestion'' (i.e. an abrupt threshold effect: if the density touches a certain maximal value, the motion is strongly affected, while nothing happens for smaller values of the density). See \cite{MauRouSan1} for comparison between the different classes of models. This last class of models, due to the discontinuity in the congestion effects, presents new mathematical difficulties, which cannot be analyzed with the usual techniques from conservation laws (or, more generally, evolution PDEs) used for soft congestion.
 
A very powerful tool to attack macroscopic hard-congestion problems is the theory of optimal transportation (see \cite{villani,OTAM}), as we can see in \cite{MauRouSan2, MauRouSan1, aude_phd, xedp}. In this framework, the density of the agents solves a continuity equation (with velocity field taking into account the congestion effects), and can be seen as a curve in the Wasserstein space. 

Our aim in this paper is to endow the macroscopic hard congestion models of  \cite{MauRouSan2, MauRouSan1, aude_phd, xedp} with diffusion effects. In other words, we will study an evolution equation where particles 
\begin{itemize}
\item have a spontaneous velocity field $u_t(x)$ which depends on time and on their position, and is the velocity they would follow in the absence of the other particles,
\item must adapt their velocity to the existence of an incompressibility constraint which prevents the density to go beyond a given threshold,
\item are subject to some diffusion effect.
\end{itemize}
This can be considered as a model for a crowd where a part of the motion of each agent is driven by a Brownian motion. Implementing this new element into the existing models could give a better approximation of reality: as usual when one adds a stochastic component, this can be a (very) rough approximation of unpredictable effects which are not already handled by the model, and this could work well when dealing with large populations.

Anyway, we do not want to discuss here the validity of this hard-congestion model and we are mainly concerned with its mathematical analysis. In particular, we will consider existence and regularity estimates, while we do not treat the uniqueness issue. Uniqueness is considered in a recent work of the first author in collaboration with S. Di Marino, see \cite{DiMMes}, and one can observe that the insertion of diffusion dramatically simplifies the picture as far as uniqueness is concerned. 

We also underline that one of the goals of the current paper (and of  \cite{DiMMes}) is to better ``prepare'' these hard congestion crowd motion models for a possible analysis in the framework of Mean Field Games (see \cite{lasry1, lasry2, lasry3}, and also \cite{modest}). These MFG models usually involve a stochastic term, also implying regularizing effects, that are useful in the mathematical analysis of the corresponding PDEs.

\subsection{The existing first order models in the light of \cite{MauRouSan2,MauRouSan1}}

Some macroscopic models for crowd motion with density constraints and ``hard congestion'' effects were studied in \cite{MauRouSan1} and \cite{MauRouSan2}. We briefly present them as follows:

\begin{itemize}
\item The density of the population in a bounded (convex) domain $\O\subset\R^d$ is described by a probability measure $\rho\in\P(\O).$ The initial density $\rho_0\in\cP(\Om)$ evolves in time, and $\rho_t$ denotes its value at each time $t\in[0,T]$.
\item The spontaneous velocity field of the population is a given time-dependent field, denoted by $u_t.$ It represents the velocity that each individual would like to follow in the absence of the others. Ignoring the density constraint, this would give rise to the continuity equation $\partial_t\rho_t+\nabla\cdot\left(\rho_t u_t\right)=0$. We observe that in the original work \cite{MauRouSan2} the vector field $u_t(x)$ was taken of the form $-\nabla D(x)$ (independent of time and of gradient form) but we try here to be more general (see \cite{aude_phd} where the non-gradient case is studied under some stronger regularity assumptions).

\item The set of admissible densities will be denoted by $\cK:=\{\rho\in\P(\O):\rho\le1\}.$ In order to guarantee that $\cK$ is neither empty nor trivial, we suppose $|\O|>1$.

\item The set of admissible velocity fields with respect to the density $\rho$ is characterized by the sign of the divergence of the velocity field on the saturated zone. We need to suppose also that all admissible velocity fields are such that no mass exists from the domain. So
formally we set $$\adm(\rho):=\left\{v:\O\to\R^d:\nabla\cdot v\ge 0\ {\rm{on}}\ \{\rho=1\}\ {\rm{and}}\ v\cdot n\le 0\ {\rm{on}}\ \partial\Om\right\}.$$

\item We consider the projection operator $P$ in  $L^2(\mathcal{L}^d)$:
$$\ds P_{\adm(\rho)}[u]\in{\rm{argmin}}_{v\in\adm(\rho)}\int_\Om|u-v|^2\dd x.$$
Note that we could have used the Hilbert space $L^2(\rho)$ instead of $L^2(\mathcal{L}^d)$: this would be more natural in this kind of evolution equations, as $L^2(\rho)$ is interpreted in a standard way as the tangent space to the Wasserstein space $\mathcal{W}_2(\Omega)$. Yet, these two projections turn out to be the same in this case, as the only relevant zone is $\{\rho=1\}$. This is just formal, and would require more rigorous definitions (in particular of the divergence constraint in $\adm(\rho)$, see below). Anyway, to clarify, we choose to use the $L^2(\mathcal{L}^d)$-projection: in this way the vector fields are considered as defined Lebesgue-a.e. on the whole $\Omega$ (and not only on $\{\rho>0\}$) and the dependence of the projected vector field on $\rho$ only passes through the set $\adm(\rho)$.
\item Finally we solve the following modified continuity equation for $\rho$
\begin{equation}\label{continuity}
\partial_t\rho_t+\nabla\cdot\left(\rho_t P_{\adm(\rho_t)}[u_t]\right)=0,
\end{equation}
where the main point is that $\rho$ is advected by a vector field, compatible with the constraints, which is the closest to the spontaneous one.
\end{itemize}

The problem in solving Equation \eqref{continuity} is that the projected field has very low regularity: it is a priori only $L^2$ in $x$, and it does not depend smoothly on $\rho$ either (since a density $1$ and a density $1-\e$ give very different projection operators). By the way, its divergence is not well-defined either. To handle this issue we need to redefine the set of admissible velocities by duality. Taking a test function $p\in H^1(\Om),\ p\ge 0\ {\rm{a.e.}}$, we obtain by the integration-by-parts equality
$$\int_\Om v\cdot\nabla p\dd x=- \int_{\Om}(\nabla\cdot v) p\dd x+\int_{\partial\Om}p v\cdot n\dd\cH^{d-1}(x).$$ 
For vector fields $v$ which do not let mass go through the boundary $\partial\Omega$ we have (in an a.e. sense) $v\cdot n=0$. This leads to the following definition
$$\adm(\rho)=\left\{v\in L^2(\Om;\R^d):\int_\O v\cdot\nabla p\ \dd x\le 0,\ \forall p\in H^1(\O),p\ge0, p(1-\rho)=0\ {\rm{a.e.}}\right\},$$
(indeed, for smooth vector field with vanishing normal component on the boundary, this is equivalent to imposing $\nabla\cdot v\geq 0$ on the set $\{\rho=1\}$).

Now, if we set 
$$\press(\rho):=\left\{p\in H^1(\Om):p\ge0,\ p(1-\rho)=0\ {\rm{a.e.}}\right\},$$ 
we observe that, by definition, $\adm(\rho)$ and $\nabla \press(\rho)$ are two convex cones which are dual to each other in $L^2(\Om;\R^d)$.
Hence we always have a unique orthogonal decomposition 
\begin{equation}\label{decomposition}
u=v+\nabla p,\quad v\in \adm(\rho);\;p\in \press(\rho),\quad \int_\O v\cdot\nabla p\dd x=0.
\end{equation}
In this decomposition (as it is the case every time we decompose on two dual convex cones), $v=P_{\adm(\rho)}[u]$. These will be our mathematical definitions for $\adm(\rho)$ and for the projection onto this cone. 


Via this approach (introducing the new variable $p$ and using its characterization from the previous line), for a given desired velocity field $u:[0,T]\times\Om\to\R^d,$ the continuity equation \eqref{continuity} can be rewritten as a system for the pair of variables $(\rho,p)$ which is 
\begin{equation}\label{syst with press}
\left\{
\begin{array}{ll}
\partial_t\rho_t+\nabla\cdot\left(\rho_t(u_t-\nabla p_t)\right)=0, &{\rm{in}}\;\; [0,T]\times\Omega,\\
p\ge0,\ \rho\leq 1,\ p(1-\rho)=0,&{\rm{in}}\;\;[0,T]\times\Omega,\\
\rho_t(u_t-\nabla p_t)\cdot n=0,&{\rm{on}}\;\;[0,T]\times\partial\Omega.
\end{array}
\right.
\end{equation}
This system is endowed with the initial condition $\rho(0,x)=\rho_0(x)$ ($\rho_0\in \cK$). As far as the spatial boundary $\partial\Omega$ is concerned, we put no-flux boundary conditions to preserve the mass in $\Omega$.

Note that in the above system we withdrew the condition $ \int (u_t-\nabla p_t)\cdot\nabla p_t=0$, as it is a consequence of the system \eqref{syst with press} itself. Informally, this can be seen in the following way: for an arbitrary $p_0\in\press(\rho_{t_0})$, we have that $t\mapsto \int_\O p_0\rho_t$ is maximal at $t=t_0$ (where it is equal to $\int_\O p_0$). Differentiating this quantity w.r.t. $t$ at $t=t_0$, using the equation \eqref{syst with press}, we get the desired orthogonality condition at $t=t_0$. For a rigorous proof of this fact (which holds for a.e. $t_0$), we refer to Proposition 4.7 in \cite{DMS}.
\subsection{A diffusive counterpart}
The goal of our work is to study a second order model of crowd movements with hard congestion effects where beside the transport factor a non-degenerate diffusion is present as well. The diffusion is the consequence of a randomness  (a Brownian motion) in the movement of the crowd. 

With the ingredients that we introduced so far, we will modify the Fokker-Planck equation $\partial_t\rho_t-\Delta\rho_t+\nabla\cdot\left(\rho_t u_t\right)=0$ in order to take into account the density constraint $\rho_t\leq 1$. Assuming enough regularity for the velocity field $u$, we observe that the Fokker Planck equation is derived from a motion given by the SODE $\dd X_t=u_t(X_t)\dd t+\sqrt{2}\dd B_t$ (where $B_t$ is the standard $d$-dimensional Brownian motion), but is macroscopically represented by the advection of the density $\rho_t$ by the vector field $-\nabla \rho_t/\rho_t+u_t$. Projecting onto the set of admissible velocities raises a natural question: should we project only $u_t$, and {\it then} apply the diffusion, or project the whole vector field, including $-\nabla \rho_t/\rho_t$? But this is not a real issue, since, at least formally, $\nabla\rho_t/\rho_t=0$ on the saturated set $\{\rho_t=1\}$ and $P_{\adm(\rho_t)}[-\nabla\rho_t/\rho_t+u_t]=P_{\adm(\rho_t)}[-\nabla\rho_t/\rho_t]+P_{\adm(\rho_t)}[u_t]=0+P_{\adm(\rho_t)}[u_t]$. Rigorously, this corresponds to the fact that the Heat Kernel preserves the constraint $\rho\leq 1$. As a consequence, we consider the modified Fokker-Planck type equation
\begin{equation}\label{fokker}
\partial_t\rho_t-\Delta\rho_t+\nabla\cdot\left(\rho_t P_{\adm(\rho_t)}[u_t]\right)=0,
\end{equation}
which can also be written equivalently for the variables $(\rho,p)$ as  
\begin{equation}\label{fokker2}
\left\{
\begin{array}{ll}
\partial_t\rho_t-\Delta\rho_t+\nabla\cdot\left(\rho_t(u_t-\nabla p_t)\right)=0, &{\rm{in}}\;\;[0,T]\times\Omega,\\
p\ge0,\ \rho\leq 1, \ p(1-\rho)=0,&{\rm{in}}\;\;[0,T]\times\Omega.
\end{array}
\right.
\end{equation}
As usual, these equations are complemented by no-flux boundary conditions and by an initial datum $ \rho(0,x)=\rho_0(x)$.

Roughly speaking, we can consider that this equation describes the law of a motion where each agent solves the stochastic differential equation
$$\dd X_t=(u_t(X_t)-\nabla p_t(X_t))\dd t+\sqrt{2}\dd B_t.$$
This last statement is just formal and there are several issues defining an SODE like this: indeed, the pressure variable is also an unknown, and globally depends on the law $\rho_t$ of $X_t$. Hence, if we wanted to see this evolution as a superposition of individual motions, each agent should somehow predict the evolution of the pressure in order to solve his own equation. This reminds of some notions from the stochastic control formulation of Mean-Field Games, as introduced by J.-M. Lasry and P.-L. Lions, even if here there are no strategic issues for the players. For MFG with density constraints, we refer to 
\cite{CarMesSan,  MesSil, modest}.

However, in this paper we will not consider any microscopic or individual problem, but only study the parabolic PDE \eqref{fokker2}.

\subsection{Structure of the paper and main results}

The main goal of the paper is to provide an existence result, with some extra estimates, for the Fokker-Planck equation \eqref{fokker2} via time discretization, using the so-called splitting method (the two main ingredients of the equation, i.e. the advection with diffusion on one hand, and the density constraint on  the other hand, are treated one after the other). In Section \ref{sec:2} we will collect some preliminary results, including what we need from optimal transport and from the previous works about density-constrained crowd motion, in particular on the projection operator onto the set $\cK$. In Section \ref{sec:main} we will provide the existence result we aim at, by a splitting scheme and some entropy bounds; the solution will be a curve of measures in $AC^2([0,T];\mathcal W_2(\O))$ (absolutely continuous curves with square-integrable speed). In Section \ref{sec:bv} we will make use of $BV$ estimates to justify that the solution we just built is also $\mathrm{Lip}([0,T];\mathcal W_1(\O))$ and satisfies a global $BV$ bound $\|\rho_t\|_{BV}\leq C$ (provided $\rho_0\in BV$): this requires to combine $BV$ estimates on the Fokker-Planck equation (which are available depending on the regularity of the vector field $u$) with $BV$ estimates on the projection operator on $\cK$ (which have been recently proven in \cite{gafb}). Section \ref{sec:5} presents a short review of alternative approaches, all discretized in time, but based either on gradient-flow techniques (the JKO scheme, see \cite{jko}) or on different splitting methods. Finally, in the Appendix \ref{sec:app} we detail the $BV$ estimates on the Fokker-Planck equation (without any density constraint) that we could find; this seems to be a delicate matter, interesting in itself, and we are not aware of the sharp assumptions on the vector field $u$ to guarantee the $BV$ estimate that we need.

\section{Preliminaries}\label{sec:2}

\subsection{Basic definitions and general facts on optimal transport}

Here we collect some tools from the theory of optimal transportation, Wasserstein spaces, its dynamical formulation, etc. which will be used later on. We set our problem either in a compact convex domain $\Om\subset\R^d$ with smooth boundary or in the $d-$dimensional flat torus $\Om:=\T^d$ (even if we will not adapt all our notations to the torus case). We refer to \cite{villani,OTAM} for more details. Given two probability measures $\mu,\nu\in \cP(\Om)$ and for $p\ge1$ we define the usual Wasserstein metric by means of the Monge-Kantorovich optimal transportation problem
$$W_p(\mu,\nu):=\inf\left\{ \int_{\Om\times\Om}|x-y|^p\dd\g(x,y)\;:\;\g\in\Pi(\mu,\nu)\right\}^{\frac1p},$$ 
where $\Pi(\mu,\nu):=\{\g\in\cP(\Om\times\Om):\;\; (\pi^x)_\#\g=\mu,\; (\pi^y)_\#\g=\nu\}$ and $\pi^x$ and $\pi^y$ denote the canonical projections from $\Om\times\Om$ onto $\Om.$  This quantity happens to be a distance on $\pical(\Om)$ which metrizes the weak-$*$ convergence of probability measures; we denote by $\mathcal W_p(\Om):=(\cP(\Om),W_p),$ i.e. the space of probabilities on $\Om$ endowed with this distance. 

Moreover, in the quadratic case $p=2$ and under the assumption $\mu\ll\cL^d$ (the $d-$dimensional Lebesgue measure on $\Om$) in the late 80's Y. Brenier showed (see \cite{brenier1, brenier2}) that actually the optimal $\ov\g$ in the above problem (the existence of which is obtained simply by the direct method of calculus of variations) is induced by a map, which is the gradient of a convex function, i.e. there exists $S:\O\to\O$ and $\psi:\Om\to\R$ convex such that $S=\nabla \psi$ and $\ov\g:=(\id,S)_\#\mu.$ The function $\psi$ is obtained as $\ds\psi(x)=\mbox{{\small $\frac 12$}} |x|^2-\varphi(x)$, where $\varphi$ is the so-called Kantorovich potential for the transport from $\mu$ to $\nu$, and is characterized as the solution of a dual problem that we will not develop here. In this way, the optimal transport map $S$ can also be written as $S(x)=x-\nabla\varphi(x)$. Later in the 90's R. McCann  (see \cite{mccann}) introduced a notion of interpolation between probability measures: the curve $\mu_t:=\left( (T-t)x+ty\right)_\#\ov\g,$ for $t\in[0,T]$ ($T>0$ is given), gives a constant speed geodesic in the Wasserstein space connecting $\mu_0:=\mu$ and $\mu_T:=\nu.$

Based on this notion of interpolation in 2000 J.-D. Benamou and Y. Brenier used some ideas from fluid mechanics to give a dynamical formulation to the Monge-Kantorovich problem (see \cite{BB}). They showed that
$$\frac{1}{pT^{p-1}}W_p^p(\mu,\nu)=\inf\left\{\cB_p(E,\mu)\; : \; \partial_t\mu+\nabla\cdot E=0,\; \mu_0=\mu,\; \mu_T=\nu \right\}.$$ Here $\cB_p$ is a functional defined on pairs $(E,\mu)$, where $E$ is a $d$-dimensional vector measure on $[0,T]\times \O$ and $\mu=(\mu_t)_t$ is a Borel-measurable family of probability measures on $\O$. This functional is defined to be finite only if $E=E_t\otimes \dd t$ (i.e. it is induced by a measurable family of vector measures on $\O$: $\int_{[0,T]\times\O}\xi(t,x)\cdot \dd E(t,x)=\int_0^T\dd t\int_\O \xi(t,x)\cdot \dd E_t(x)$ for all test functions $\xi\in C^0([0,T]\times\O;\R^d)$) and in this case it is defined through
$$
\cB_p(E,\mu):=\left\{
\begin{array}{ll}
\ds\int_0^T\int_\Om \frac{1}{p}\left|v_t\right|^p\dd \mu_t(x)\dd t, &{\rm{if}}\ E_t=v_t\cdot\mu_t\,\\
+\infty, & {\rm{otherwise}}.
\end{array}
\right. 
$$
It is well-known that $\cB_p$ is jointly convex and l.s.c. w.r.t the weak-$*$ convergence of measures (see Section 5.3.1 in \cite{OTAM}) and that, if $\partial_t \mu+\nabla\cdot E=0$, then $\cB_p(E,\mu)<+\infty$ implies that $t\mapsto \mu_t$ is a curve in $AC^{p}([0,T];\mathcal W_p(\Omega))$\footnote{Here $AC^p([0,T];\mathcal W_p(\Om))$ denotes the class of absolutely continuous curves in $\mathcal W_p(\Om)$ with metric derivative in $L^p$. See the connection with the functional $\cB_p$.}. In particular it is a continuous curve and the initial and final conditions on $\mu_0$ and $\mu_T$ are well-defined.

Coming back to curves in Wasserstein spaces, it is well known (see \cite{ags} or Section 5.3 in \cite{OTAM}) that for any distributional solution $\mu_t$ (being a continuous curve in $\mathcal W_p(\Om)$) of the continuity equation $\partial_t\mu+\div E=0$ with $E_t=v_t\cdot\mu_t$, we have the relations 
$$|\mu'|_{W_p}(t)\le\|v_t\|_{L^p_{\mu_t}}\;\;\;{\rm{and}}\;\;\; W_p(\mu_t,\mu_s)\le\int_s^t|\mu'|_{W_p}(\t)\dd\t,$$
where we denoted by $|\mu'|_{W_p}(t)$ the metric derivative w.r.t. $W_p$ of the curve $\mu_t$ (see for instance \cite{AmbTil} for general notions about curves in metric spaces and their metric derivative). For curves $\mu_t$ that are geodesics in $\mathcal W_p(\Om)$ we have the equality
$$W_p(\mu_0,\mu_1)=\int_0^1|\mu'|_{W_p}(t)\dd t=\int_0^1\|v_t\|_{L^p_{\mu_t}}\dd t.$$
The last equality is in fact the Benamou-Brenier formula with the optimal velocity field $v_t$ being the density of the optimal $E_t$ w.r.t. the optimal $\mu_t.$ This optimal velocity field $v_t$ can be computed as $v_t:=(S-\id)\circ (S_t)^{-1}$, where $S_t:=(1-t)\id+tS$ is the transport in McCann's interpolation (we assume here that the initial measure $\mu_0$ is absolutely continuous, so that we can use transport maps instead of plans). This expression can be obtained if we consider that in this interpolation particles move with constant speed $S(x)-x$, but $x$ represents here a Lagrangian coordinate, and not an Eulerian one: if we want to know the velocity at time $t$ at a given point, we have to find out first the original position of the particle passing through that point at that time.

In the sequel we will also need the notion of entropy of a probability density, and for any probability measure $\varrho\in\cP(\Om)$ we define it as
$$\ds\cE(\varrho):=\left\{
\begin{array}{ll}
\ds\int_\Omega\varrho(x)\log\varrho(x)\dd x, & \rm{if}\ \varrho\ll\cL^d,\\
+\infty, & \rm{otherwise}.
\end{array}
\right.$$
We recall that this functional is l.s.c. and geodesically convex in $\mathcal W_2(\Om)$.

As we will be mainly working with absolutely continuous probability measures (w.r.t. Lebesgue), we often identify measures with their densities. 

\subsection{Projection problems in Wasserstein spaces}\label{subsec:proj}
Our analysis strongly relies on the projection operator $P_\cK$ in the sense of $W_2.$ Here $\cK:=\{\rho\in\P(\O):\rho\le1\}$ and 
$$P_\cK[\mu]:=\argmin_{\rho\in\cK}\;\frac12 W_2^2(\mu,\rho).$$
We recall (see \cite{MauRouSan2,xedp} and \cite{gafb}) the main properties of the projection $P_\cK$ operator.
\begin{itemize}
\item As far as $\Omega$ is compact, for any probability measure $\mu$, the minimizer in $\min_{\rho\in\cK}\frac12 W_2^2(\mu,\rho)$ exists and is unique, and the operator $P_\cK$ is continuous (it is even $C^{0,1/2}$ for the $W_2$ distance).
\item The projection $P_\cK[\mu]$ saturates the constraint $\rho\leq 1$ in the sense that for any $\mu\in\cP(\Om)$ there exists a measurable set $B\subseteq\Om$ such that $P_\cK[\mu]=\one_B+\mu^{\rm{ac}}\one_{B^c},$
where $\mu^{\rm{ac}}$ is the absolutely continuous part of $\mu$.
\item The projection is characterized in terms of a pressure field, in the sense that $\rho=P_\cK[\mu]$ if and only if there exists a Lipschitz function $p\geq 0$, with $p(1-\rho)=0$, and such that the optimal transport map $S$ from $\rho$ to $\mu$ is given by $S:=\id-\nabla\vphi=\id+\nabla p$.
\item There is (as proven in \cite{gafb}) a quantified $BV$ estimate: if $\mu\in BV$ (in the sense that it is absolutely continuous and that its density belongs to $BV(\Omega)$), then $P_\cK[\mu]$ is also $BV$ and $$TV(P_\cK[\mu],\Om)\le TV(\mu,\Om).$$
\end{itemize}

This last $BV$ estimate will be crucial in Section \ref{sec:bv}, and it is important to have it in this very form (other estimates of the form $TV(P_\cK[\mu],\Om)\le aTV(\mu,\Om)+b$ would not be as useful as this one, as they cannot be easily iterated). 

%

\section{Existence via a splitting-up type algorithm {\it \textbf{ (Main Scheme)}}}\label{sec:main}
 Similarly to the approach in \cite{MauRouSan1} (see the algorithm (13) and Theorem 3.5) for a general, non-gradient, vector field, we will build a theoretical algorithm, after time-discretization, to produce a solution of  \eqref{fokker2}. Let us remark that splitting-type methods have been widely used in other contexts as well, see for instance the paper \cite{CleMaas} which deals with splitting methods for Fokker-Planck equations and for more general gradient flows in metric and Wasserstein spaces, or \cite{Laborde} where a splitting-like approach is used to attack PDEs which are not gradient flows but ``perturbations'' of gradient flows.
 
In this section the spontaneous velocity field is a general vector field $u:[0,T]\times\O\to\R^d$ (not necessarily a gradient), which depends also on time. The only assumption we require on $u$ is the following:
\begin{equation}\label{hyp:U}
u\in L^\infty([0,T]\times\Om;\R^d).\tag{U}
\end{equation}
We will work on a time interval $[0,T]$ and in a bounded convex domain $\Om\subset\R^d$ (the case of the flat torus is even simpler and we will not discuss it in details). We consider $\rho_0\in\cP^{\rm{ac}}(\Om)$ to be given, which represents the initial density of the population, and we suppose $\rho_0\in\cK$.
 
\subsection{Splitting using the Fokker-Planck equation}

Let us consider the following scheme. 
\smallskip

\begin{minipage}{8.1cm}
{\textbf{Main scheme:}}
Let $\t>0$ be a small time step with $N:=\lfloor T/\t\rfloor.$ Let us set $\rho_0^\t:=\rho_0$ and for every $k\in\{1,\dots,N\}$ we define $\rho_{k+1}^\tau$ from $\rho_k^\tau$ in the following way. First we solve
\begin{equation}\label{FP-basic}
\left\{
\begin{array}{l}
\partial_t\varrho_t -\Delta\varrho_t+\nabla\cdot(\varrho_t u_{t+k\t})=0, \ t\in]0,\t],\\
\varrho_{0}=\rho_k^\tau,
\end{array}
\right.
\end{equation}
equipped with the no-flux boundary condition ($\varrho_t(\nabla\varrho_t-u_{t})\cdot n =0$ a.e. on $\partial\Om$)
and set $\rho_{k+1}^\tau=P_\cK[\tilde{\rho}_{k+1}^\tau],$ where $\tilde{\rho}_{k+1}^\tau=\varrho_\tau.$ See Figure \ref{cxc} on the right.
\end{minipage}
\begin{minipage}{6.5cm}
\begin{tikzpicture}[scale=0.6]
\draw[blue] (0,0) node {$\bullet$} node [left]{$\rho^\tau_k$};
\draw[blue,->,>=stealth] (0.5,0.3) -- (7.5,5.5) node[midway,above,sloped] {$\partial_t\varrho_t-\Delta\varrho_t+\nabla\cdot\left(\varrho_t u_{t+k\t}\right)=0$};
\draw[blue] (8,6) node {$\bullet$} node [above]{$\tilde\rho^\tau_{k+1}=\varrho_{\t}$};
\draw[blue,->,>=stealth] (9.7,0.3) -- (8.15,5.2) node[midway,above,sloped] {$\id+\t\nabla p_{k+1}^\t$};
\draw[blue] (10,0) node {$\bullet$} node [right]{$\rho^\tau_{k+1}$};
\end{tikzpicture}

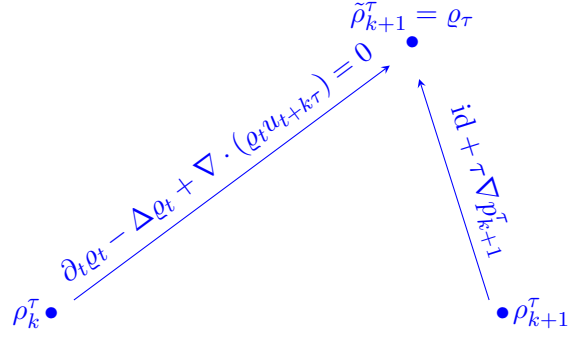
\captionof{figure}{One time step}\label{cxc}
\end{minipage}
\smallskip 

Let us remark first that by classical results on parabolic equations (see for instance \cite{lady}), since $u$ satisfies the assumption \eqref{hyp:U}, Problem \ref{FP-basic} admits a unique distributional solution. 

The above algorithm means the following: first follow the Fokker-Planck equation, ignoring the density constraint, for a time $\tau$, then project. In order to state and prove the convergence of the scheme, we need to define some suitable interpolations of the discrete sequence of densities that we have just introduced.

{\it First interpolation.} We define the following curves of densities, velocities and momenta constructed with the help of the $\rho_k^\tau$'s. First set
$$
\rho^\t_t:=\left\{
\begin{array}{ll}
\varrho_{2(t-k\t)}, & {\rm{if}}\ t\in\left[k\t,(k+1/2)\t\right[,\\
\left(\id+2((k+1)\t-t)\nabla p_{k+1}^\t\right)_\#\rho_{k+1}^\t, &  {\rm{if}}\ t\in\left[(k+1/2)\t,(k+1)\t\right[,
\end{array}
\right.
$$
where $\varrho_t$ is the solution of the Fokker-Planck equation \eqref{FP-basic} with initial datum $\rho_k^\t$ and $\nabla p_{k+1}^\t$ arises from the projection of $\tilde\rho_{k+1}^\t,$ more precisely $(\id+\t\nabla p_{k+1}^\t)$ is the optimal transport from $\rho_{k+1}^\t$ to $\tilde\rho_{k+1}^\t.$ What are we doing? We are fitting into a time interval of length $\tau$ the two steps of our algorithm. First we follow the FP equation \eqref{FP-basic} at double speed, then we interpolate between the measure we reached and its projection following the geodesic between them. This geodesic is easily described as an image measure of $\rho_{k+1}^\t$ through McCann's interpolation. By the construction it is clear that $\rho^\t_t$ is a continuous curve in $\cP(\Om)$ for $t\in[0,T].$ We now define a family of time-dependent vector fields though
$$
v^\t_t:=\left\{
\begin{array}{ll}
-2\frac{\nabla{\varrho_{2(t-k\t)}}}{\varrho_{2(t-k\t)}}+2u_{t}, & {\rm{if}}\ t\in\left[k\t,(k+1/2)\t\right[,\\
-2\nabla p_{k+1}^\t\circ(\id+2((k+1)\t-t)\nabla p_{k+1}^\t)^{-1}, &  {\rm{if}}\ t\in\left[(k+1/2)\t,(k+1)\t\right[,
\end{array}
\right.
$$
and, finally, we simply define the curve of momenta as $E_t^\t:=\rho_t^\t v_t^\t.$ 


{\it Second interpolation.} We define another interpolation as follows. Set
$$\tilde\rho_t^\t:=\varrho_{t-k\t},\;\;\; {\rm{if}}\ t\in[k\t,(k+1)\t[,$$
where $\varrho_t$ is (again) the solution of the Fokker-Planck equation \eqref{FP-basic} on the time interval $[0,\t]$ with initial datum $\rho_k^\t.$ Here we do not double its speed. We define the curve of velocities 
$$\tilde v_t^\t:=-\frac{\nabla{\varrho_{t-k\t}}}{\varrho_{t-k\t}}+u_{t},\;\;\; {\rm{if}}\ t\in\left[k\t,(k+1)\t\right[,$$
and we build the curve of momenta by $\tilde E_t^\t:=\tilde\rho_t^\t \tilde v_t^\t.$ 

{\it Third interpolation.} For each $\t$, we also define piecewise constant curves, 
\begin{eqnarray*}\hat\rho_t^\t:=\rho_{k+1}^\t,\;\;\; &&{\rm{if}}\ t\in[k\t,(k+1)\t[,\\
\hat v_t^\t:=\nabla p_{k+1}^\t,\;\;\; &&{\rm{if}}\ t\in[k\t,(k+1)\t[,\end{eqnarray*}
and $\hat E_t^\t:=\hat\rho_t^\t\hat v_t^\t.$ We remark that $p_{k+1}^\t(1-\rho_{k+1}^\t)=0,$ hence the curve of momenta is just
$$\hat E_t^\t:=\nabla p_{k+1}^\t,\;\;\; {\rm{if}}\ t\in[k\t,(k+1)\t[.$$

Mind the differences in the construction of $\rho_t^\t,$ $\tilde\rho_t^\tau$ and $\hat\rho_t^\tau$ (hence in the construction of $v_t^\t$, $\tilde v_t^\t$ and $\hat v_t^\t$ and $E_t^\t$, $\tilde E_t^\t$ and $\hat E_t^\t$): 

1) the first one is continuous in time for the weak-* convergence, while the second and third ones are not;

2) in the first construction we have taken into account the projection operator explicitly, while in the second one we see it just in an indirect manner (via the `jumps' occurring at every time of the form $t=k\t$). The third interpolation is piece-wise constant, and at every time it satisfies the density constraint;

3) in the first interpolation the pair $(\rho^\tau,E^\tau)$ solves the continuity equation, while in the other two they do not. This is not astonishing, as the continuity equation characterizes continuous curves in $\mathcal W_2(\O)$.

In order to prove the convergence of the scheme above, we will obtain uniform $AC^2([0,T];\mathcal W_2(\Om))$ bounds for the curves $\rho^\t.$ A key observation here is that the metric derivative (w.r.t. $W_2$) of the solution of the Fokker-Planck equation is comparable with the time differential of the entropy functional along the same solution (see Lemma \ref{entropy-FP}). Now we state the main theorem of this section.

\begin{theorem}\label{convergence}
Let $\rho_0\in\cK$ and $u$ be a given desired velocity field satisfying \eqref{hyp:U}. Let us consider the interpolations introduced above. Then there exists a continuous curve $[0,T]\ni t\mapsto \rho_t\in\mathcal W_2(\Om)$ and some vector measures  $E,\tilde E, \hat E\in\fM([0,T]\times\Om)$ such that the curves $\rho^\t,\tilde\rho^\t,\hat\rho^\t$ converge uniformly in $\mathcal W_2(\Om)$  to $\rho$ and 
$$
E^\t\weaks E,\quad \tilde E^\t\weaks \tilde E,\quad \hat E^\t\weaks \hat E,\;\; {\rm{in}}\;\; \fM([0,T]\times\Om)^d,\;\; {\rm{as}}\ \t\to0.
$$
Moreover $E=\tilde E-\hat E$ and for a.e. $t\in[0,T]$ there exist time-dependent measurable vector fields $v_t,\tilde v_t,\hat v_t$ such that 
\begin{itemize}
\item[(1)] $E=\rho v,\tilde E=\rho\tilde v,\;\; \hat E=\rho\hat v,$ 
\item[(2)] $\ds\int_0^T\left(\|v_t\|_{L^2_{\rho_t}}^2+\|\tilde v_t\|_{L^2_{\rho_t}}^2+ \|\hat v_t\|_{L^2_{\rho_t}}^2\right)\dd t<+\infty,$
\item[(3)] $v_t=\tilde v_t-\hat v_t\ \rho_t-{\rm{a.e.}},\;\; \tilde E_t=\rho_t u_t-\nabla\rho_t\;\; {\rm{and}}\;\; \hat v_t=\nabla p_t,\ \rho_t-{\rm{a.e.}},$
\end{itemize}
where $p\in L^2([0,T];H^1(\Om)),$ $p\geq 0$ and $p(1-\rho)=0$ a.e. in $[0,T]\times\Om$. As a consequence, the pair $(\rho,p)$ is a weak solution of the problem
\begin{equation}\label{FP-advanced}
\left\{
\begin{array}{ll}
\partial_t\rho_t-\Delta\rho_t+\nabla\cdot\left(\rho_t(u_t-\nabla p_t)\right)=0, & {\rm{in}}\ [0,T]\times\Om,\\
p_t\ge 0,\ \rho_t\leq 1,\ p_t(1-\rho_t)=0, & {\rm{in}}\ [0,T]\times\Om,\\
\rho_t(\nabla\rho_t-u_t+\nabla p_t)\cdot n = 0, & {\rm{on}}\ [0,T]\times\partial\Om,\\
\rho(0,\cdot)=\rho_0.
\end{array}
\right.
\end{equation}
\end{theorem}
To prove this theorem we will use the following tools.

\begin{lemma}\label{entropy-FP}
Let us consider a solution $\varrho_t$ of the Fokker-Planck equation on $[0,T]\times\Om$ with the velocity field $u$ satisfying \eqref{hyp:U} and with no-flux boundary conditions on $[0,T]\times\partial\Om$. Then for any time interval $]a,b[$ we have the following estimate 
\begin{equation}\label{estim-velocity}
\frac{1}{2}\int_a^b\int_{\Om}\left|-\frac{\nabla\varrho_t}{\varrho_t}+u_{t}\right|^2\varrho_t\dd x\dd t\le\cE(\varrho_a)-\cE(\varrho_b)+\frac{1}{2}\int_a^b\int_\Om|u_{t}|^2\varrho_t\dd x\dd t
\end{equation}
In particular this implies
\begin{equation}\label{tool2}
\frac{1}{2}\int_a^b|\varrho'_t|^2_{W_2}\dd t\le\cE(\varrho_a)-\cE(\varrho_b)+\frac{1}{2}\int_a^b\int_\Om|u_{t}|^2\varrho_t\dd x\dd t,
\end{equation}
where $|\varrho_t'|_{W_2}$ denotes the metric derivative of the curve $t\mapsto\varrho_t\in\mathcal W_2(\Om)$.
\end{lemma}
\begin{proof}
To prove this inequality, we will first make computations in the case where both $u$ and $\varrho$ are smooth, and $\varrho$ is boundeded from below by a positive constant. In this case we can write
\begin{align*}
\frac{\dd}{\dd t}\cE(\varrho_t)&=\int_\Om(\log\varrho_t+1)\partial_t\varrho_t\dd x=\int_\Om \log\varrho_t(\Delta\varrho_t-\nabla\cdot(\varrho_t u_{t}))\dd x\\
&=\int_\Om\left(-\frac{|\nabla\varrho_t|^2}{\varrho_t}+u_{t}\cdot\nabla\varrho_t\right)\dd x,
\end{align*}
where we used the conservation of mass (i.e. $\int_\Om \partial_t\varrho_t\dd x=0$) and the boundary conditions in the integration by parts.
We now compare this with
\begin{eqnarray*}
\frac 12\int_{\Om}\left|-\frac{\nabla\varrho_t}{\varrho_t}+u_{t}\right|^2\varrho_t\dd x-\frac12\int_\Om|u_{t}|^2\varrho_t\dd x&=&\int_\Om\left(\frac12\frac{|\nabla\varrho_t|^2}{\varrho_t}-\nabla\varrho_t\cdot u_{t}\right)\dd x\\&\le&\int_\Om\left(\frac{|\nabla\varrho_t|^2}{\varrho_t}-\nabla\varrho_t\cdot u_{t}\right)\dd x =-\frac{\dd}{\dd t}\cE(\varrho_t).
\end{eqnarray*}
This provides the first part of the statement, i.e. \eqref{estim-velocity}. If we combine this with the fact that the metric derivative of the curve $t\mapsto\varrho_t$ is always less or equal than the $L^2_{\varrho_t}$ norm of the velocity field in the continuity equation, we also get
$$\frac12|\varrho'_t|^2_{W_2}-\frac12\int_\Om|u_{t}|^2\varrho_t\le-\frac{\dd}{\dd t}\cE(\varrho_t),$$
and hence \eqref{tool2}.

In order to prove the same estimates without artificial smoothness and lower bound assumptions, we can act by approximation. We approximate the density $\varrho_a$ by smooth and strictly positive densities $\varrho^k_a$ (by convolution, so that we guarantee in particular $\cE(\varrho_a^k)\to\cE(\varrho_a)$), and the vector field $u$ with smooth vector fields $u^k$ (strongly in $L^4([a,b]\times\Omega)$, keeping the $L^\infty$ bound). If we call $\varrho^k$ the corresponding solution of the Fokker Planck equation, it satisfies \eqref{estim-velocity}. This implies a uniform bound (w.r.t. $k$) for $\sqrt{\varrho^k}$ in $L^2([a,b];H^1(\Omega))$, and hence a uniform bound on $\varrho^k$ in $L^2([a,b]\times\Omega)$. From these bounds and the uniqueness of the solution of the Fokker-Planck equation with $L^\infty$ drift we deduce $\varrho^k\to\varrho$. The semicontinuity of the left-hand side in \eqref{estim-velocity} and of the entropy term at $t=b$, together with the convergence of the entropy at $t=a$ and the convergence $\int_a^b\int_\O |u^k|^2\varrho^k\dd x\dd t\to\int_a^b\int_\O |u|^2\varrho\dd x\dd t$ (because we have a product of weak and strong convergence in $L^2$) allow to pass \eqref{estim-velocity} to the limit.
\end{proof}

\begin{corollary}\label{entropy_small}
From the inequality \eqref{tool2} we deduce that 
$$\cE(\varrho_b)-\cE(\varrho_a)\le\frac{1}{2}\int_a^b\int_\Om|u_{t}|^2\varrho_t\dd x\dd t,$$
hence in particular for $u$ satisfying \eqref{hyp:U} we have
$$\cE(\varrho_b)-\cE(\varrho_a)\le\frac12 \|u\|^2_{L^\infty}(b-a).$$
As a consequence, if $\ro_a\le 1,$ then we have
$$\cE(\varrho_b)\le\frac12 \|u\|^2_{L^\infty}(b-a).$$
The same estimate can be applied to the curve $\tilde\rho^\t$, with $a=k\t$ and $b\in ]k\t,(k+1)\t[$, thus obtaining $\cE(\tilde\rho^\t_t)\leq C\t$ for every $t$.
\end{corollary}

\begin{lemma}\label{proj_entropy}
For any $\rho\in\cP(\Om)$ we have $\cE\left(P_\cK[\rho]\right)\le\cE(\rho).$ 
\end{lemma}
\begin{proof}
We can assume $\rho\ll\cL^d$, otherwise the claim is straightforward. As we pointed out in Section \ref{subsec:proj},  we know that there exists a measurable set $B\subseteq\Om$ such that 
$$P_\cK[\rho]=\one_B+\rho\one_{B^c}.$$ Hence it is enough to prove that 
$$\int_B\rho\log\rho\dd x\ge 0=\int_B P_\cK[\rho]\log P_\cK[\rho]\dd x,$$
as the entropies on $B^c$ coincide. As the mass of $\rho$ and $P_\cK[\rho]$ is the same on the whole $\Omega$, and they coincide on $B^c$, we have $\ds\int_B\rho(x)\dd x=\int_B P_\cK[\rho]\dd x=|B|$.

Then, by Jensen's inequality we have  
$$\frac{1}{|B|}\int_B\rho\log\rho\dd x\ge \left(\frac{1}{|B|}\int_B\rho\dd x\right)\log\left(\frac{1}{|B|}\int_B\rho\dd x\right)=0.$$
The entropy decay follows. 
\end{proof}
To analyze the pressure field we will need the following result.

\begin{lemma}\label{pressure}
Let $\{p^\t\}_{\t>0}$ be a bounded sequence in $L^2([0,T]; H^1(\Om))$ and $\{\rho^\t\}_{\t>0}$ a sequence of piecewise constant curves valued in $\mathcal W_2(\Om)$, which satisfy $W_2(\rho^\t(a),\rho^\t(b))\leq C\sqrt{b-a+\t}$ for all $a<b\in [0,T]$ for a fixed constant $C$. Suppose that 
$$p^\tau\geq 0,\; p^\tau(1- \rho^\tau)=0,\;\rho^\tau\leq 1,$$
and that
$$p^\tau\weak p\; {\rm{weakly\ in}}\; L^2([0,T]; H^1(\Om))\;\;{\rm{and}}\;\; \rho^\tau\to\rho\; {\rm{uniformly\ in\ }}\mathcal W_2(\Om).$$
Then $p(1-\rho)=0$ a.e. in $[0,T]\times\Om$.
\end{lemma}
The proof of this result is the same as in Step 3 of Section 3.2 of \cite{MauRouSan2} (see also \cite{aude_phd} and Lemma 4.6 in \cite{DMS}). We omit it in order not to overburden the paper. 

The reader can note the strong connection with the classical Aubin-Lions lemma \cite{Aubin}, applied to the compact injection of $L^2$ into $H^{-1}$. Indeed, from the weak convergence of $p^\t$ in  $L^2([0,T]; H^1(\Om))$, we just need to provide strong convergence of $\rho^\t$ in $L^2([0,T]; H^{-1}(\Om))$. If instead of the quasi-H\"older assumption of the above lemma we suppose a uniform bound of  $\{\rho^\t\}_\t$ in $AC^2([0,T];\mathcal W_2(\O))$ (which is not so different), then the statement can be really deduced from the Aubin-Lions lemma. Indeed, the sequence  $\{\rho^\t\}$ is bounded in $L^\infty([0,T]; L^2(\Om))$ and its time-derivative would be bounded in $L^2([0,T]; H^{-1}(\Om))$. This strongly depends on the fact that the $H^{-1}$ distance can be controlled by the $W_2$ distance as soon as the measures have uniformly bounded densities (see \cite{loeper,MauRouSan2}), a tool which also crucial in the proofs in \cite{MauRouSan2,aude_phd,DMS}. Then, the Aubin-Lions lemma guarantees compactness in $C^0([0,T];H^{-1}(\O))$, which is more than what we need. 

\begin{lemma}\label{tools}
Let us consider the previously defined interpolations. Then we have the following facts.
\begin{itemize}
\item[{$\rm{(i)}$}] For every $\t>0$ and $k$ we have 
$$\max\left\{W_2^2(\rho_k^\t,\tilde\rho_{k+1}^\t), W_2^2(\rho_{k}^\t,\rho_{k+1}^\t)\right\}\le \t C\left(\cE(\rho_k^\t)-\cE(\rho_{k+1}^\t)\right)+C\t^2,$$
where $C>0$ only depends on $\|u\|_{L^\infty}.$ 

\item[(ii)] There exists a constant $C$, only depending on $\rho_0$ and $\|u\|_{L^\infty}$, such that $\cB_2(E^\t,\rho^\t)\le C$, $\cB_2(\tilde E^\t,\tilde\rho^\t)\le C$ and $\cB_2(\hat E^\t,\hat\rho^\t)\le C$.

\item[(iii)] For the curve $[0,T]\ni t\mapsto\rho_t^\t$ we have that  $$\int_0^T|(\rho_t^\tau)'|^2_{W_2}\dd t\le C,$$ for a $C>0$ independent of $\t$. Here we denoted by $|(\rho_t^\tau)'|_{W_2}$ the metric derivative of the curve $\rho^\tau$ at $t$ in $\mathcal W_2$. In particular, we have a uniform H\"older bound on $\rho^\tau$: $W_2(\rho^\tau(a),\rho^\tau(b))\leq C\sqrt{b-a}$ for every $b>a$.

\item[(iv)] $E^\t,\tilde E^\t, \hat E^\t$ are uniformly bounded sequences in $\fM([0,T]\times\Om)^d.$  
\end{itemize}
\end{lemma}
\begin{proof}
${\rm{(i)}}$ First by the triangle inequality and by the fact that $\rho_{k+1}^\tau=P_\cK[\tilde{\rho}_{k+1}^\t]$ we have that
\begin{equation}\label{tool1}
W_2(\rho_k^\t,\rho_{k+1}^\t)\le W_2(\rho_k^\t,\tilde{\rho}_{k+1}^\t)+W_2(\tilde{\rho}_{k+1}^\t,\rho_{k+1}^\t)\le2W_2(\rho_k^\t,\tilde{\rho}_{k+1}^\t).
\end{equation}

We use (as before) the notation $\varrho_t,$ $t\in[0,\t]$ for the solution of the Fokker-Planck equation \eqref{FP-basic} with initial datum $\rho_k^\t,$ in particular we have $\varrho_\t=\tilde\rho_{k+1}^\t.$ 
Using Lemma \ref{entropy-FP} and since $\varrho_0=\rho_k^\tau$ and $\varrho_\t=\tilde\rho_{k+1}^\t$ we have by \eqref{tool2} and using $\ds W_2(\rho_k^\t,\tilde\rho_{k+1}^\t)\le\int_0^\t|\varrho_t'|_{W_2}\dd t$
\begin{align*}
W_2^2(\rho_k^\t,\tilde{\rho}_{k+1}^\t)&\le\left(\t^{\frac12}\left(\int_0^\t|\varrho_t'|^2_{W_2}\dd t\right)^{\frac12}\right)^2\le 2\t\left(\cE(\varrho_0)-\cE(\varrho_\t)\right)+\t\int_0^\t\int_\Om|u_{k\t+t}|^2\varrho_t\dd x\dd t\\
&\le 2\t\left(\cE(\rho_k^\t)-\cE(\tilde\rho_{k+1}^\t)\right)+C\t^2\le 2\t\left(\cE(\rho_k^\t)-\cE(\rho_{k+1}^\t)\right)+C\t^2,
\end{align*}
where $C>0$ is a  constant depending just on $\|u\|_{L^\infty}$. We also used the fact that $\cE(\rho_{k+1}^\t)\le\cE(\tilde\rho_{k+1}^\t)$, a consequence of Lemma \ref{proj_entropy}.

Now by the means of \eqref{tool1} we obtain
$$W_2^2(\rho_k^\t,\rho_{k+1}^\t)\le \t C\left(\cE(\rho_k^\t)-\cE(\rho_{k+1}^\t)\right)+C\t^2.$$

${\rm{(ii)}}$ We use Lemma \ref{entropy-FP} on the intervals of type $[k\t,(k+1/2)\t[$ and the fact that on each interval of type $[(k+1/2)\t,(k+1)\t[$ the curve $\rho_t^\t$ is a constant speed geodesic. In particular, on these intervals we have  
$$|(\rho^\t)'|_{W_2}=\|v^\t_t\|_{L^2_{\rho^\t_t}}=2\t\|\nabla p_{k+1}^\t\|_{L^2_{\rho_{k+1}^\t}}=2W_2(\rho_{k+1}^\t,\tilde\rho_{k+1}^\t).$$
On the other hand we also have
$$\t^2\|\nabla p_{k+1}^\t\|_{L^2_{\rho_{k+1}^\t}}^2=W_2^2(\rho_{k+1}^\t,\tilde\rho_{k+1}^\t)\le W_2^2(\rho_{k}^\t,\tilde\rho_{k+1}^\t)\le \t C\left(\cE(\rho_k^\t)-\cE(\rho_{k+1}^\t)\right)+C\t^2.$$
Hence we obtain
\begin{align*}
&\int_{k\t}^{(k+1)\t}\|v^\t_t\|^2_{L^2(\rho^\t_t)}\dd t\\
&=\int_{k\t}^{(k+1/2)\t}\int_\Om 4\left|-\frac{\nabla\varrho_{2(t-k\t)}}{\varrho_{2(t-k\t)}}+u_{2t-k\t}\right|^2\varrho_{2(t-k\t)}(x)\dd x\dd t+4\int_{(k+1/2)\t}^{(k+1)\t}\int_\Om|\nabla p_{k+1}^\t|^2\rho_{k+1}^\t\dd x\dd t\\
&\le C\left(\cE(\rho_k^\t)-\cE(\rho_{k+1}^\t)\right)+C\t + 2\t\|\nabla p_{k+1}^\t\|^2_{L^2_{\rho_{k+1}^\t}}\\
&\le C\left(\cE(\rho_k^\t)-\cE(\rho_{k+1}^\t)\right)+C\t.
\end{align*}
Hence by adding up we obtain 
$$\cB_2(E^\tau,\rho^\tau)\le \sum_k \left\{C\left(\cE(\rho_k^\t)-\cE(\rho_{k+1}^\t)\right)+C\t\right\}=C\left(\cE(\rho_0^\t)-\cE(\rho_{N+1}^\t)\right)+CT\le C.$$

The estimate on $\cB_2(\tilde E^\tau,\tilde\rho^\tau)$ and $\cB_2(\hat E^\tau,\hat \rho^\tau)$ are completely analogous and descend from the previous computations.

${\rm{(iii)}}$ The estimate on $\cB_2(E^\tau,\rho^\tau)$ implies a bound on $\ds\int_0^T|(\rho_t^\tau)'|^2_{W_2}\dd t$ because $v^\t$ is a velocity field for $\rho^\t$ (i.e., the pair $(E^\tau,\rho^\tau)$ solves the continuity equation).

${\rm{(iv)}}$ In order to estimate the total mass of $E$ we write
\begin{align*}
|E^\t|([0,T]\times\Om)&=\int_0^T\int_\Om|v_t^\t|\rho_t^\t\dd x\dd t\le\int_0^T\left(\int_\Om |v_t^\t|^2\rho_t^\t\dd x\right)^\frac12\left(\int_\Om\rho_t^\t\dd x\right)^\frac12\dd t\\
&\le \sqrt{T}\left(\int_0^T\int_\Om|v_t^\t|^2\rho_t^\t\dd x\dd t\right)^\frac12\le C.
\end{align*}  
The bounds on $\tilde E^\t$ and $\hat E^\t$ rely on the same argument. 
\end{proof}

\begin{proof}[Proof of Theorem \ref{convergence}]
We use the tools from Lemma \ref{tools}. 

{\it Step 1.}
By the bounds on the metric derivative of the curves $\rho_t^\t$ we get compactness, i.e. there exists a curve $[0,T]\ni t\mapsto\rho_t\in\cP(\Om)$ such that $\rho^\t$ (up to subsequences) converges uniformly in $[0,T]$ w.r.t. $W_2,$ in particular weakly-$*$ in $\cP(\Om)$ for all $t\in[0,T].$ It is easy to see that $\tilde\rho^\t$ and $\hat\rho^\t$ are converging to the same curve. Indeed we have $\tilde\rho_t^\t=\rho_{\tilde s(t)}^\t$ and $\hat\rho_t^\t=\rho_{\hat s(t)}^\t$ for $|\tilde s(t)-t|\leq \t$ and $|\hat s(t)-t|\leq \t$, which implies  $W_2(\rho_t^\t,\tilde\rho_t^\t),W_2(\rho_t^\t,\hat\rho_t^\t) \le C\t^\frac12$. This provides the convergence to the same limit.

{\it Step 2.} By the boundedness of $E^\t,\tilde E^\t$ and $\hat E^\t$ in $\fM([0,T]\times\Om)^d$ we have the existence of $E,\tilde E,\hat E\in\fM([0,T]\times\Om)^d$ such that (up to a subsequence) $E^\t\weaks E,\tilde E^\t\weaks \tilde E, \hat E^\t\weaks \hat E$ as $\t\to 0.$ Now we show that $E=\tilde E-\hat E.$ Indeed, let us show that for any test function $f\in \Lip([0,T]\times\Om)^d$ we have 
$$\left|\int_0^T\int_\Om f_t\cdot \left(E_t^\t-(\tilde E_t^\t+\hat E_t^\t)\right)(\dd x,\dd t)\right|\to 0,$$
as $\t\to 0.$ First for each $k\in\{0,\dots,N\}$ we  have that
\begin{eqnarray*}
\int_{k\t}^{(k+1/2)\t}\int_\Om f_t \cdot E_t^\t(\dd x,\dd t)&=&\int_{k\t}^{(k+1)\t}\int_\Om f_{(t+k\t)/2}\cdot(-\nabla \varrho_{t-k\t}+u_{t}\varrho_{t-k\t})(\dd x,\dd t)\\
&=&\int_{k\t}^{(k+1)\t}\int_\Om f_t \cdot \tilde E_t^\t(\dd x,\dd t)+\int_{k\t}^{(k+1)\t}\int_\Om \left( f_{(t+k\t)/2}-f_t\right)\cdot\tilde E_t^\t(\dd x,\dd t) \end{eqnarray*}
and
\small
\begin{eqnarray*}
\int_{(k+1/2)\t}^{(k+1)\t}\int_\Om f_t \cdot E_t^\t(\dd x,\dd t)&=&\int_{k\t}^{(k+1)\t}\int_\Om -f_{(t+(k+1)\t)/2}\circ(\id+((k+1)\t-t)\nabla p_{k+1}^\t)\cdot\nabla p_{k+1}^{\t}\rho_{k+1}^\t(\dd x,\dd t)\\
&=&-\int_{k\t}^{(k+1)\t}\int_\Om f_t \cdot \hat E_t^\t(\dd x,\dd t)\\
&&+\int_{k\t}^{(k+1)\t}\int_\Om \left(f_t-f_{(t+(k+1)\t)/2}\circ(\id+((k+1)\t-t))\right)\cdot\hat v^\t_t\hat \rho^\t_t(\dd x,\dd t)\end{eqnarray*}
\normalsize
This implies that
\begin{align*}
\Bigg{|}\int_0^T\int_\Om f_t\cdot (E_t^\t&-\tilde E_t^\t+\hat E_t^\t)(\dd x,\dd t)\Bigg{|} \le \sum_k\int_{k\t}^{(k+1)\t}\Lip(f)\t\int_\Om |\tilde E^\tau_t|(\dd x,\dd t)\\
&+\sum_k\int_{k\t}^{(k+1)\t}\Lip(f)\tau\int_\Om (1+|\hat v^\t_t|)|\hat E_t^\t|(\dd x,\dd t)\\
&\le \t C\Lip(f) \left(|\tilde E^\t|([0,T]\times\Om)+ |\hat E^\t|([0,T]\times\Om)+\cB_2(\hat E,\hat\rho)\right)\\
&\le  \t C\Lip(f),
\end{align*}
for a uniform constant $C>0$. Letting $\t\to 0$ we prove the claim. 

{\it Step 3.} The bounds on $\cB_2(E^\t,\rho^\t), \cB_2(\tilde E^\t,\tilde\rho^\t)$ and $\cB_2(\hat E^\t,\hat\rho^\t)$ pass to the limit by semicontinuity and allow to conclude that $E, \tilde E$ and $\hat E$ are vector valued Radon measures absolutely continuous w.r.t. $\rho.$  Hence there exist $v_t,\tilde v_t,\hat v_t$ such that $E=\rho v$, $\tilde E = \rho \tilde v$ and $\hat E=\rho\hat v.$ 

{\it Step 4.} We now look at the equations satisfied by $E, \tilde E$ and $\hat E$. First we use $\partial_t \rho^\t+\nabla\cdot E^\t=0$, we pass to the limit as $\t\to 0$, and we get 
$$\partial_t \rho+\nabla\cdot E=0.$$

Then, we use $\tilde E^\t=-\nabla\tilde\rho^\t+u_t\tilde\rho^\t$, we pass to the limit again as $\t\to 0$, and we get 
$$\tilde E=-\nabla\rho+u_t\rho.$$
To justify the above limit, the only delicate point is passing to the limit the term $u_t\tilde\rho^\t$, since $u$ is only $L^\infty$, and $\tilde\rho^\t$ converges weakly as measures, and we are a priori only allowed to multiply it by continuous functions. Yet, we remark that by Corollary \ref{entropy_small} we have that $\cE(\tilde\rho_t^\t)\le C\t$ for all $t\in[0,T]$. In particular, this provides, for each $t$, uniform integrability for $\tilde\rho_t^\t$ and turns the weak convergence as measures into weak convergence in $L^1$. This allows to multiply by $u_t$ in the weak limit.

Finally, we look at $\hat E^\t$. There exists a piecewise constant (in time) function $p^\t$ (defined as $p_{k+1}^\t $ on every interval $]k\t,(k+1)\t]$) such that $p^\t\geq 0$, $p^\t(1-\hat\rho^\t)=0$, 
\begin{equation}\label{L2H1bound p}
\int_0^T\int_\Om |\nabla p^\t|^2 (\dd x,\dd t)=\int_0^T\int_\Om |\nabla p^\t|^2 \hat\rho^\t (\dd x,\dd t)=\int_0^T\int_\Om |\hat v^\t|^2 \hat\rho^\t (\dd x,\dd t)\leq C
\end{equation}
and $\hat E^\t=\nabla p^\t \hat\rho^\t=\nabla p^\t $. The bound \eqref{L2H1bound p} implies that $p^\t$ is uniformly bounded in 
$L^2(0,T;H^1(\Om))$. Since for every $t$ we have $|\{p^\t_t=0\}|\geq |\{\hat\rho^\t_t<1\}|\geq |\Om|-1$, we can use a suitable version of Poincar\'e's inequality, and get a uniform bound in $L^2([0,T];L^2(\Om))=L^2([0,T]\times \Om)$. Hence there exists $p\in L^2([0,T]\times \Om)$ such that $p^\t\weak p$ weakly in $L^2$ as $\t\to 0.$ In particular we have $\hat E=\nabla p$. Moreover it is clear that $p\ge 0$ and by Lemma \ref{pressure} we obtain $p(1-\rho)=0$ a.e. as well. Indeed, the assumptions of the Lemma are easily checked: we only need to estimate $W_2(\hat\rho^\t(a),\hat\rho^\t(b))$ for $b>a$, but we have
$$W_2(\hat\rho^\t(a),\hat\rho^\t(b))=W_2(\rho^\t(k_a\t),\rho^\t(k_b \t))\leq C\sqrt{k_b-k_a},\quad\mbox{ for $k_b\t\leq b+\t$ and $k_a\geq a$}.$$
Once we have $\hat E=\nabla p$ with $p(1-\rho)=0$, $p\in L^2([0,T];H^1(\Om))$ and $\rho\in L^\infty$, we can also write 
$$\hat E=\nabla p=\rho\nabla p.$$

If we sum up our results, using $E=\tilde E-\hat E$, we have
$$\partial_t \rho -\Delta\rho+\nabla\cdot (\rho(u-\nabla p))=0\;\;\;\mbox{together with }p\geq 0,\,\rho\leq 1,\,p(1-\rho)=0\;\;{\rm{a.e.\ in\ }}[0,T]\times\Omega.$$
As usual, this equation is satisfied in a weak sense, with no-flux boundary conditions.
\end{proof}

\section{Uniform $\Lip([0,T];\mathcal W_1)$ and $BV$ estimates}\label{sec:bv}

In this section we provide uniform estimates for the curves $\rho^\t,\tilde\rho^\t$ and $\hat\rho^\t$ of the following form: we prove uniform $BV$ (in space) bounds on $\tilde\rho^\t$ (which implies the same bound for $\hat\rho^\t$) and uniform Lipschitz bounds in time for the $W_1$ distance on $\rho^\t$. 
This means a small improvement compared to the previous section in what concerns time regularity, as we have Lipschitz instead of $AC^2$, even if we need to replace $W_2$ with $W_1$. It is also important in what concerns space regularity. Indeed, from Lemma \ref{entropy-FP} one could deduce that the solution $\rho$ of the FP equation \eqref{fokker2} satifies $\sqrt{\rho}\in L^2([0,T];H^1(\O))$ and, using $\rho\leq 1$, also $\rho\in L^2([0,T];H^1(\O))$. Yet, this is just an integrable estimate in $t$, while the $BV$ estimate of this section is uniform in the time variable.

Nevertheless there is a price to pay for this improvement: we have to assume higher regularity for the velocity field. These uniform-in-time $W_1$-Lipschitz bounds are based both on $BV$ estimates for the Fokker-Planck equation (see Lemma \ref{bv_estimate1} from Appendix A) and for the projection operator $P_\cK$ (see \cite{gafb}). The assumption on $u$ is essentially the following: we need to control the growth of the total variation of the solutions of the Fokker-Planck equation \eqref{FP-basic}, and we need to iterate this bound along time steps.

We will discuss in the Appendix the different $BV$ estimates on the Fokker-Planck equation that we were able to find. The desired estimate is true whenever $\|u_t\|_{C^{1,1}(\Om)}$ is uniformly bounded  $u_t\cdot n=0$ on $\partial\Om$. It seems to be an open problem to obtain similar estimate under the only assumption that $u$ is Lipschitz continuous. Of course, we will also assume $\rho_0\in BV(\Omega)$. Despite these extra regularity assumptions, we think these estimates have their own interest, exploiting some finer properties of the solutions of the Fokker-Planck equation and of the Wasserstein projection operator. 

Before entering into the details of the estimates, we want to discuss why we concentrate on $BV$ estimates (instead of Sobolev ones) and on $W_1$ (instead of $W_p$, $p>1$). The main reason is the role of the projection operator: indeed, even if $\rho\in W^{1,p}(\Om)$, we do not have in general $P_\cK[\rho]\in W^{1,p}$ because the projection creates some jumps at the boundary of $\{P_\cK[\rho]=1\}$. This prevents from obtaining any $W^{1,p}$ estimate for $p>1$. On the other hand, \cite{gafb} exactly proves a $BV$ estimate on $P_\cK[\rho]$ and paves the way to $BV$ bounds for our equation. Concerning the regularity in time, we observe that the velocity field in the Fokker-Planck equation contains a term in $\nabla\rho/\rho$. Since the metric derivative in $\mathcal W_p$ is given by the $L^p$ norm (w.r.t. $\rho_t$) of the velocity field, it is clear that estimates in $\mathcal W_p$ for $p>1$ would require spatial $W^{1,p}$ estimates on the solution itself, which are impossible for $p>1$ in this splitting scheme. We underline that this does not mean that uniform $W^{1,p}$ are impossible for the solution of \eqref{fokker2}; it only means that they are not uniform along the approximation that we used in our \emph{Main Scheme} to build such a solution.

The precise result that we prove is the following.
\begin{theorem}\label{lip-estimate}
Let us suppose that $\|u_t\|_{C^{1,1}}\leq C$ and $\rho_0\in BV(\Om)$. Then using the notations from the \emph{Main scheme} and Theorem \ref{convergence} one has $\|\tilde\rho^\t_t\|_{BV}\leq C$ and $W_1(\rho^\t_k,\rho^\t_{k+1})\leq C\t$. As a consequence we also have $\rho\in\Lip([0,T];\mathcal W_1)\cap L^\infty([0,T];BV(\Om))$.
\end{theorem}
To prove this theorem we need the following lemmas.
\begin{lemma}\label{tool3}
Suppose $\|u_t\|_{\Lip}\leq C$ and $u_t\cdot n=0$ on $\partial\Om$. Then for the solution $\varrho$ of \eqref{FP-app} with velocity field $v=u$ we have the estimate
$$\|\varrho_t\|_{L^\infty}\le\|\varrho_0\|_{L^\infty} e^{Ct},$$
where $C=\|\nabla\cdot u_t\|_{L^\infty}$.
\end{lemma}

\begin{proof}
Standard comparison theorems for parabolic equations allow to prove the results once we notice that $f(t,x):=\|\varrho_0\|_{L^\infty} e^{Ct}$ is a supersolution of the Fokker-Planck equation, i.e.
$$\partial_tf_t\geq \Delta f_t -\nabla\cdot(f_t u_t).$$
Indeed, in the above equation the Laplacian term vanishes as $f$ is constant in $x$, $\partial_t f_t=Cf_t$ and $\nabla\cdot(f_t u_t)=f_t\nabla\cdot u_t+\nabla f_t\cdot u_t=f_t\nabla\cdot u_t\leq Cf_t$ where  $C=\|\nabla\cdot u_t\|_{L^\infty}$.
From this inequality, and from $\rho_0\leq f_0$, we deduce $\rho_t\leq f_t$ for all $t$.
\end{proof}
We remark that the above lemma implies in particular that after every step in the {\it Main scheme} we have $\tilde\rho_{k+1}^\t\le e^{\t c}\leq 1+C\t,$ where $c:=\|\nabla\cdot u\|_{L^\infty}.$ Let us now present the following lemma as well.

\begin{corollary}\label{tool4}
Along the iterations of our ${\rm{Main\ scheme}}$, for every $k$ we have $W_1(\tilde\rho_{k+1}^\t,\rho_{k+1}^\t)\le \t C$ for a constant $C>0$ independent of $\tau$. 
\end{corollary}

\begin{proof}
With the saturation property of the projection (see Section \ref{subsec:proj} or \cite{gafb}), we know that there exists a measurable set $B\subseteq\Om$ such that  $\rho_{k+1}^\t=\tilde\rho_{k+1}^\t\one_B+\one_{\Om\setminus B}.$ On the other hand we know that 
\begin{eqnarray*}
W_1(\tilde\rho_{k+1}^\t,\rho_{k+1}^\t)&=&\sup_{f\in\Lip_{1}(\Omega),\,0\leq f\leq \rm{diam}(\Omega)}\int_\Om f(\tilde\rho_{k+1}^\t-\rho_{k+1}^\t)\dd x\\
&=&\sup_{f\in\Lip_{1}(\Omega),\,0\leq f\leq \rm{diam}(\Omega)}\int_{\Om\setminus B}\! f(\tilde\rho_{k+1}^\t-1)\dd x\le\t C\,|\Om|\rm{diam}(\Om).
\end{eqnarray*}
We used the fact that the competitors $f$ in the dual formula can be taken positive and bounded by the diameter of $\Om$, just by adding a suitable constant. This implies as well that $C$ is depending on $c,|\Om|$ and ${\rm{diam}(\Om)}.$
\end{proof}

\begin{proof}[Proof of Theorem \ref{lip-estimate}]

First we take care of the $BV$ estimate. Lemma \ref{bv_estimate1} in the Appendix guarantees, for $t\in ]k\t,(k+1)\t[,$ that we have $TV(\tilde\rho^\t_t)\leq C\t+e^{C\t} TV(\rho^\t_k)$. Together with the $BV$ bound on the projection that we presented in Section \ref{subsec:proj} (taken from \cite{gafb}), this can be iterated, providing a uniform bound (depending on $TV(\rho_0)$, $T$ and $\sup_t \|u_t\|_{C^{1,1}}$) on $\|\tilde\rho^\t_t\|_{BV}$. Passing this estimate to the limit as $\t\to 0$ we get $\rho\in L^\infty([0,T];BV(\Om))$.

Then we estimate the behavior of the interpolation curve $\hat\rho^\tau$ in terms of $W_1$. We estimate
\begin{align*}
W_1(\rho_k^\t,\tilde\rho_{k+1}^\t)\le \int_{k\t}^{(k+1)\t}|(\tilde\rho_t^\t)'|_{W_1}\dd t
&\le\int_{k\t}^{(k+1)\t}\int_\Om\left(\frac{|\nabla\tilde\rho^\t_{t}|}{\tilde\rho^\t_{t}}+|u_t|\right)\tilde\rho^\t_{t}\dd x\dd t\\
&\le \int_{k\t}^{(k+1)\t}\|\tilde\rho^\t_{t}\|_{BV}\dd t+C\t\le C\t.\end{align*}

Hence, we obtain
$$
W_1(\rho_k^\t,\rho_{k+1}^\t)\le W_1(\rho_k^\t,\tilde\rho_{k+1}^\t)+W_1(\tilde\rho_{k+1}^\t,\rho_{k+1}^\t)\le \t C.
$$

This in particular means, for $b>a$,
$$W_1(\hat\rho^\t(a),\hat\rho^\t(b))\leq C(b-a+\t).$$
We can pass this relation to the limit, using that, for every $t$, we have $\hat\rho^\t_t\to \rho_t$ in $\mathcal W_2(\Om)$ (and hence also in $\mathcal W_1(\Om)$, since $W_1\leq W_2$), we get 
$$W_1(\rho(a),\rho(b))\leq C(b-a),$$
which means that $\rho$ is Lipschitz continuous in $\mathcal W_1(\Om)$.
\end{proof}

\section{Variations on a theme: some reformulations of the {\it\textbf{Main scheme}}}\label{sec:5}

In this section we propose some alternative approaches to study the problem  \eqref{fokker2}. The general idea is to discretize in time, and give a way to produce a measure $\rho^\t_{k+1}$ starting from $\rho^\t_k$. Observe that the interpolations that we proposed in the previous sections $\rho^\t, \tilde\rho^\t$ and $\hat\rho^\t$ are only technical tools to state and prove a convergence result, and the most important point is exactly the definition of $\rho^\t_{k+1}$.

The alternative approaches proposed here explore different ideas, more difficult to implement than the one that we presented in Section \ref{sec:main}, and/or restricted to some particular cases (for instance when $u$ is a gradient). They have their own modeling interest and this is the main reason justifying their sketchy presentation.

\subsection{Variant 1: transport, diffusion then projection.} 
We recall that the original splitting approach for the equation without diffusion (\cite{MauRouSan1,aude_phd}) exhibited an important difference compared to what we did in Section \ref{sec:main}. Indeed, in the first phase of each time step (i.e. before the projection) the particles follow the vector field $u$ and $\tilde\rho^\t_{k+1}$ was not defined as the solution of a continuity equation with advection velocity given by $u_t$, but as the image of $\rho^\t_k$ via a straight-line transport: $\tilde\rho^\t_{k+1}:=(\id+\t u_{k\t})_\#\rho^\t_k$. One can wonder whether it is possible to follow a similar approach here.

A possible way to proceed is the following: take a random variable $X$ distributed according to $\rho^\t_k$, and define $\tilde\rho^\t_{k+1}$ as the law of $X+\t u_{k\t}(X)+B_\t$, where $B$ is a Brownian motion, independent of $X$. This exactly means that every particle moves starting from its initial position $X$, following a displacement ruled by $u$, but adding a stochastic effect in the form of the value at time $\t$ of a Brownian motion. We can check that this means
$$\tilde{\rho}_{k+1}^\t:=\eta_\tau*\left((\id+\t u_{k\t})_\#\rho^\t_k\right),$$
where $\eta_\tau$ is a Gaussian kernel with zero-mean and variance $\t$, i.e. $\ds\eta_\tau(x):=\frac{1}{(4\tau\pi)^{d/2}}e^{-\frac{|x|^2}{4\tau}}.$ 

Then we define 
$$\rho_{k+1}^\t:=P_\cK\left[\tilde{\rho}_{k+1}\right].$$

Despite the fact that this scheme is very natural and essentially not that different from the {\it Main scheme}, we have to be careful with the analysis. First we have to quantify somehow the distance $W_p(\rho_k^\t,\tilde\rho_{k+1}^\t)$ for some $p\ge1$ and show that this is of order $\t$ in some sense. Second, we need to be careful when performing the convolution with the heat kernel (or adding the Brownian motion, which is the same): this requires either to work in the whole space (which was not our framework) or in a periodic setting ($\Om=\T^d$, the flat torus, which is qutie restrictive). Otherwise, the ``explicit'' convolution step should be replaced with some other construction, such as following the Heat equation (with Neumann boundary conditions) for a time $\tau$. But this brings back to a situation very similar to the {\it Main scheme}, with the additional difficulty that we do not really have estimates on $(\id+\t u_{k\t})_\#\rho^\t_k$.

\subsection{Variant 2: gradient flow techniques for gradient velocity fields}

In this section we assume that the velocity field of the population is given by the opposite of the gradient of a function, $u_t=-\nabla V_t$ a typical example is given when we take for $V$ the distance function to the exit (see the discussions in \cite{MauRouSan2} about this type of question). We start from the case where $V$ does not depend on time, and we suppose $V\in W^{1,1}(\Om)$. In this particular case -- beside the splitting approach -- the problem has a variational structure, hence it is possible to show the existence by the means of gradient flows in Wasserstein spaces.

Since the celebrated paper of Jordan, Kinderlehrer and Otto (\cite{jko}) we know that the solutions of the Fokker-Planck equation (with a gradient vector field) can be obtained with the help of the gradient flow of a perturbed entropy functional with respect to the Wasserstein distance $W_2.$ This formulation of the JKO scheme was also used in \cite{MauRouSan2} for the first order model with density constraints. It is easy to combine the JKO scheme with density constraints to study the second order/diffusive model. As a slight modification of the model from \cite{MauRouSan2}, we can consider the following discrete implicit Euler (or JKO) scheme. As usual, we fix a time step $\tau>0,$ $\rho_0^\tau=\rho_0$ and for all $k\in\{1,2,\dots,\lfloor N/\tau\rfloor\}$ we just need to define $\rho_{k+1}^\tau$. We take
\begin{equation}\ds
\rho_{k+1}^\tau=\argmin_{\rho\in\P(\O)}\left\{\int_\O V(x)\rho(x)\dd x+\cE(\rho)+I_\cK(\rho)+\frac{1}{2\tau}W_2^2(\rho,\rho_k^\tau)\right\},
\end{equation}
where $I_\cK$ is the indicator function of $\cK,$ which is
$$I_\cK(x):=\left\{\begin{array}{ll}
0, & \rm{if}\ x\in \cK,\\
+\infty, & \rm{otherwise}.
\end{array}
\right.$$
The usual techniques from \cite{jko,MauRouSan2} can be used to identify that System \eqref{fokker2} is the gradient flow of the functional  $\ds\rho\mapsto J(\rho):=\int_\O V(x)\rho(x)\dd x+\cE(\rho)+I_\cK(\rho)$ and that the above discrete scheme converges (up to a subsequence) to a solution of \eqref{fokker2}, thus proving existence. The key estimate for compactness is
$$\frac{1}{2\tau}W_2^2(\rho^\t_{k+1},\rho_k^\tau)\leq J(\rho^\t_k)-J(\rho^\t_{k+1}),$$
which can be summed up (as on the r.h.s. we have a telescopic series), thus obtaining the same bounds on $\cB_2$ that we used in Section \ref{sec:main}.

Note that whenever $D^2V\geq \lambda I$, the functional $\rho\mapsto \int_\O V(x)\rho(x)\dd x+\cE(\rho)+I_\cK(\rho)$ is $\lambda$-geodesically convex. This allows to use the theory in \cite{ags} to prove not only existence, but also uniqueness for this equation, and even stability (contractivity or exponential growth on the distance between two solutions) in $\mathcal W_2$. Yet, we underline that the techniques of \cite{DiMMes} also give the same result. Indeed, \cite{DiMMes} contains two parts. In the first part, the equation with density constaints for a given velocity field $u$ is studied, under the assumption that $-u$ has some monotonicity properties: $(-u_t(x)+u_t(y))\cdot(x-y)\geq \lambda|x-y|^2$ (which is the case for the gradients of $\lambda$-convex functions). In this case standard Gr\"onwall estimates on the $W_2$ distance between two solutions are proved, and it is not difficult to add diffusion to that result (as the Heat kernel is already contractant in $\mathcal W_2$). In the second part, via different techniques (mainly using the adjoint equation, and proving somehow $L^1$ contractivity), the uniqueness result is provided for arbitrary $L^\infty$ vector fields $u$, but with the crucial help of the diffusion term in the equation.

It is also possible to study a variant where $V$ depends on time. We assume for simplicity that $V\in\Lip([0,T]\times \Om)$ (this is a simplification; less regularity in space, such as $W^{1,1}$, could be sufficient). In this case we define
 $$J_t(\rho):=\int_\O V_t(x)\rho(x)\dd x+\cE(\rho)+I_\cK(\rho)$$
 and 
\begin{equation}\ds
\rho_{k+1}^\tau=\argmin_{\rho\in\P(\O)}\left\{J_{k\t}(\rho)+\frac{1}{2\tau}W_2^2(\rho,\rho_k^\tau)\right\},
\end{equation}
The analysis proceeds similarly, with the only exception that the we get
$$\frac{1}{2\tau}W_2^2(\rho^\t_{k+1},\rho_k^\tau)\leq J_{k\t}(\rho^\t_k)-J_{k\t}(\rho^\t_{k+1}),$$
which is no more a a telescopic series. Yet, we have $J_{k\t}(\rho^\t_{k+1})\geq J_{(k+1)\t}(\rho^\t_{k+1})+\Lip(V)\t$, and we can go on with a telescopic sum plus a remainder of the order of $\t$. In the case where $u_t$ is the opposite of the gradient of a $\lambda$-convex function $V_t$, one could consider approximation by functions which are piecewise constant in time and use the standard theory of gradient flows.
%
%
%
%
%
%

Let us remark here that the recent paper \cite{AleKimYao} gives another approach to deal with first order crowd motion models as limit of nonlinear-diffusion equations with gradient drift. This approach could be plausible also in the case when we add a simple diffusion term in the models studied in \cite{AleKimYao}.

\subsection{ Variant 3: transport then gradient flow-like step with the penalized entropy functional.} We present now a different scheme, which combines some of the previous approaches. It could formally provide a solution of the same equation, but presents some extra difficulties. 

We define now $\tilde{\rho}_{k+1}^\tau:=(\id+\tau u_{k\tau})_\#\rho_k^\tau$ and with the help of this we define
$$\rho_{k+1}^\tau:=\argmin_{\rho\in\cK}\cE(\rho)+\frac{1}{2\tau}W_2^2(\rho,\tilde{\rho}_{k+1}^\tau).$$
In the last optimization problem we minimize a strictly convex and l.s.c. functionals, and hence we have existence and uniqueness of the solution. The formal reason for this scheme being adapted to the equation is that we perform a step of a JKO scheme in the spirit of \cite{jko} (without the density constraint) or of \cite{MauRouSan2} (without the entropy term). This should let a term $-\Delta\rho-\nabla\cdot(\rho\nabla p)$ appear in the evolution equation. The term $\nabla\cdot(\rho u)$ is due to the first step (the definition of $\tilde\rho^\t_{k+1}$). To explain a little bit more for the unexperienced reader, we consider the optimality conditions for the above minimization problem. Following \cite{MauRouSan2}, we can say that $\rho\in \cK$ is optimal if and only if there exists a constant $\ell\in\R$ and a Kantorovich potential $\varphi$ for the transport from $\rho$ to $\rho_k^\tau$ such that 

\noindent\begin{minipage}{8cm}
$$\rho=\begin{cases}1 & \mbox{ on }\left(\ln\rho+\frac\varphi\tau\right) <\ell,\\
					0 & \mbox{ on }\left(\ln\rho+\frac\varphi\tau\right) >\ell,\\
					\in[0,1] & \mbox{ on }\left(\ln\rho+ \frac\varphi\tau\right) = \ell.
\end{cases}$$
We then define $p=(\ell-\ln\rho-\frac\varphi\tau)_+$ and we get $p\in\press(\rho)$. Moreover, 
$\rho-\mbox{a.e.}\,\nabla p = -\frac{\nabla\rho}{\rho}-\frac{\nabla\varphi}{\tau}.$ We then use the fact that the optimal transport is of the form $T=\id-\nabla\varphi$ and obtain a situation as is sketched in Figure \ref{cxc2}.
\end{minipage}
\begin{minipage}{8cm}
\begin{center}
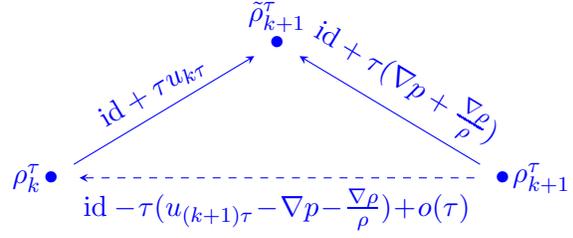

\begin{tikzpicture}[scale=0.6]
\draw[blue] (0,0) node {$\bullet$} node [left]{$\rho^\tau_k$};
\draw[blue,->,>=stealth] (0.5,0.3) -- (4.5,2.7) node[midway,above,sloped] {$\id+\tau u_{k\tau}$};
\draw[blue] (5,3) node {$\bullet$} node [above]{$\tilde\rho^\tau_{k+1}$};
\draw[blue,->,>=stealth] (9.5,0.3) -- (5.5,2.7) node[midway,above,sloped] {$\id+\tau (\nabla p+\frac{\nabla\rho}{\rho})$};
\draw[blue] (10,0) node {$\bullet$} node [right]{$\rho^\tau_{k+1}$};
\draw[blue,dashed,->,>=stealth] (9.375,0) -- (0.625,0) node[midway,below] {$\id-\!\tau (u_{(k+1)\tau}\!-\!\nabla p\!-\!\frac{\nabla\rho}{\rho})\!+\!o(\tau)$};
\end{tikzpicture}
\captionof{figure}{One time step}\label{cxc2}
\end{center}
\end{minipage}

Notice that $(\id+\tau u_{k\tau})^{-1}\circ(\id+\tau (\nabla p+\nabla\rho/\rho))=\id-\tau (u_{(k+1)\tau}-\nabla p-\nabla\rho/\rho)+o(\tau)$ provided $u$ is regular enough. Formally we can pass to the limit $\tau\to 0$ and have 
$$\partial_t\rho-\Delta\rho+\nabla\cdot(\rho(u-\nabla p))=0.$$
Yet, this turns out to be quite na\"ive, because we cannot get proper estimates on $W_2(\rho_k^\t,\rho_{k+1}^\t)$. Indeed,  this is mainly due to the hybrid nature of the scheme, i.e. a gradient flow for the diffusion and the projection part on one hand and a free transport on the other hand. The typical estimate in the JKO scheme comes from the fact that one can bound $W_2(\rho_k^\t,\rho_{k+1}^\t)^2/\t$ with the opposite of the increment of the energy, and that this gives rise to a telescopic sum. Yet, this is not the case whenever the base point for a new time step is not equal to the previous minimizer.  Moreover, the main difficulty here is the fact that the energy we consider implicitly takes the value $+\infty$, due to the constraint $\rho\in\cK$, and hence no estimate is possible whenever $ \tilde\rho^\tau_{k+1}\notin \cK$. As a possible way to overcome this difficulty, one could approximate the discontinuous functional $I_\cK$ with some finite energies of the same nature (for instance power-like entropies, even if the best choice would be an energy which is Lipschitz for the distance $W_2$). These kinds of difficulties are matter of current study, in particular for mixed systems and/or multiple populations.

\appendix
\section{$BV$-type estimates for the Fokker-Planck equation}\label{sec:app}

Here we present some Total Variation ($TV$) decay results (in time) for the solutions of the Fokker-Planck equation. Some are very easy, some trickier.  The goal is to look at those estimates which can be easily iterated in time and combined with the decay of the $TV$ via the projection operator, as we did in Section \ref{sec:bv}.

Let us take a vector field  $v:[0,+\infty[\times\Om\to\R^d$ (we will choose later which regularity we need) and consider in $\Om$ the problem 
\begin{equation}\label{FP-app}
\left\{
\begin{array}{ll}
\partial_t\rho_t -\Delta\rho_t+\nabla\cdot(\rho_tv_t)=0, & {\rm{in}}\ ]0,+\infty[\times\Om,\\
\rho_t(\nabla\rho_t-v_t)\cdot n=0, & {\rm{on}}\ [0,+\infty[\times\partial\Om,\\
\rho(0,\cdot)=\rho_0, & {\rm{in}}\ \Om,\\
\end{array}
\right.
\end{equation}
for $\rho_0\in BV(\Om)\cap\cP(\Om).$

\begin{lemma}\label{bv_estimate1}
Suppose $\|v_t\|_{C^{1,1}}\leq C$ for all $t\in[0,+\infty[.$ Suppose that either $\Om=\T^d$, or that $\Omega$ is convex and $v\cdot n=0$ on $\partial\Om$. Then, we have the following total variation decay estimate 
\begin{equation}\label{tv-estim}
\int_\Om |\nabla\rho_t|\dd x\le C(t-s)+ e^{C(t-s)}\int_\Om |\nabla\rho_s|\dd x,\;\;\;\forall\ 0\le s\le t,
\end{equation}
where $C>0$ is a constant depending just on the $C^{1,1}$ norm of $v$.\end{lemma}
\begin{proof}
First we remark that by the regularity of $v$ the quantity $\|v\|_{L^\infty}+\|D v\|_{L^\infty}+\|\nabla(\nabla\cdot v)\|_{L^\infty}$ is uniformly bounded. Let us drop now the dependence on $t$ in our notation and calculate in coordinates
\small
\begin{align*}
\frac{\dd}{\dd t}\int_\Om|\nabla\rho|\dd x&=\int_\Om\frac{\nabla\rho}{|\nabla\rho|}\cdot\nabla(\partial_t\rho)\dd x=\int_\Om\frac{\nabla\rho}{|\nabla\rho|}\cdot\nabla(\Delta\rho-\nabla\cdot(v\rho))\dd x=\int_\Om\sum_j\frac{\rho_j}{|\nabla\rho|}\left(\sum_i\rho_{iij}-(\nabla\cdot (v\rho))_j\right)\dd x\\ \normalsize
&=-\int_\Om\sum_{i,j,k}\left(\frac{\rho_{ij}^2}{|\nabla\rho|}-\frac{\rho_j\rho_k\rho_{ki}\rho_{ij}}{|\nabla\rho|^3}\right)\dd x+B_1-\int_\Om\sum_{j,i}\frac{\rho_j}{|\nabla\rho|}\left(v_{ij}^i\rho+v_{i}^i\rho_j+v^i_j\rho_i+v^i\rho_{ij}\right)\dd x\\
&\le B_1 + C + C\int_\Om|\nabla\rho|\dd x +\int_\Om|\nabla\rho||\nabla\cdot v|\dd x + B_2\\
&\le B_1 + B_2 + C + C\int_\Om|\nabla\rho|\dd x.
\end{align*}\normalsize
Here the $B_i$ are the boundary terms, i.e. 
$$\ds B_1:=\int_{\partial\Om}\sum_{i,j}\frac{\rho_j n^i\rho_{ij}}{|\nabla\rho|}\dd\cH^{d-1}\;\mbox{ and }\;\ds B_2:=-\int_{\partial\Om}(v\cdot n)|\nabla\rho|\dd\cH^{d-1}.$$ 
The constant $C>0$ only depends on $\|v\|_{L^\infty}+\|\nabla\cdot v\|_{L^\infty}+\|\nabla(\nabla\cdot v)\|_{L^\infty}.$ We used as well the fact that $\ds -\int_\Om\sum_{i,j,k}\left(\frac{\rho_{ij}^2}{|\nabla\rho|}-\frac{\rho_j\rho_k\rho_{ki}\rho_{ij}}{|\nabla\rho|^3}\right)\dd x\le 0.$

Now, it is clear that in the case of the torus the boundary terms $B_1$ and $B_2$ do not exist, hence we conclude by Gr\"onwall's lemma. In the case of the convex domain we have $B_2=0 $ (because of the assumption $v\cdot n=0$) and $B_1\leq 0$ because of the next Lemma \ref{lemma exball}.
\end{proof}

\begin{lemma}\label{lemma exball}
Suppose that $u:\Omega\to\R^d$ is a smooth vector field with $u\cdot n=0$ on $\partial\Omega$, $\rho$ is a smooth function with $\nabla\rho\cdot n=0$ on  $\partial\Omega$, and that $\Omega\subset\R^d$ is a smooth convex set that we write as $\Omega=\{h<0\}$ for a smooth convex function $h$ with $|\nabla h|=1$ on $\partial\Omega$ (so that $n=\nabla h$ on $\partial\Omega$). Then we have, on the whole boundary $\partial\Omega$,
$\ds\sum_{i,j}u^i_j \rho_j n^i=-\sum_{i,j}u^i h_{ij}\rho_j .$ 

In particular, we have $\ds\sum_{i,j}\rho_{ij}\rho_j n^i\leq 0$.\end{lemma}
\begin{proof} 

The Neumann boundary assumption on $u$ means $u(\gamma(t))\cdot \nabla h(\gamma(t))=0$ for every curve $\gamma$ valued in $\partial\Omega$ and for all $t$. Differentiating in $t$, we get
$$\sum_{i,j}u^i_j(\gamma(t))(\gamma'(t))^j h_i(\gamma(t))+\sum_{i,j}u^i(\gamma(t))h_{ij}(\gamma(t))(\gamma'(t))^j=0.$$
Take a point $x_0\in\partial\Omega$ and choose a curve $\gamma$ with $\gamma(t_0)=x_0$ and $\gamma'(t_0)=\nabla \rho(x_0)$ (which is possible, since this vector is tangent to $\partial\Omega$ by assumption). This gives the first part of the statement. The second part, i.e. $\ds\sum_{i,j}\rho_{ij}\rho_j n^i\leq 0$, is obtained by taking $u=\nabla \rho$ and using that $D^2h(x_0)$ is a positive definite matrix.
\end{proof}

%
%
%
%
\begin{remark}
If we look attentively at the proof of Lemma \ref{bv_estimate1}, we can see that we did not really exploit the regularizing effects of the diffusion term in the equation. This means that the regularity estimate that we provide are the same that we would have without diffusion: in this case, the density $\rho_t$ is obtained from the initial density as the image through the flow of $v$. Thus, the density depends on the  determinant of the Jacobian of the flow, hence on the derivatives of $v$. It is normal that, if we want $BV$ bounds on $\rho_t$, we need assumptions on two derivatives of $v$.
\end{remark}
We would like to prove some form of $BV$ estimates under weaker regularity assumptions on $v$, trying to exploit the diffusion effects. In particular, we would like to treat the case where $v$ is only $C^{0,1}$. As we will see in the following lemma, this degenerates in some sense. 

\begin{lemma}\label{bv_estimate2}
Suppose that $\Omega$ is either the torus or a smooth convex set $\Omega=\{h<0\}$ parameterized as a level set of a smooth convex function $h$.
Let $v_t:\Omega\to\R^d$ be a vector field for $t\in[0,T]$, Lipschitz and bounded in space, uniformly in time. In the case of a convex domain, suppose $v\cdot n=0$ on $\partial\Omega$. Let $H:\R^d\to\R$ be given by $H(z):=\sqrt{\e^2+|z|^2.}$ 
Now let $\rho_t$  (sufficiently smooth) be the solution of the Fokker-Planck equation with homogeneous Neumann boundary condition.

Then there exists a constant $C>0$ (depending on $v$ and $\Omega$) such that 
\begin{equation}
\int_\Omega H(\nabla\rho_t)\dd x\le \int_\Omega H(\nabla\rho_0)\dd x +C\e t +\frac C\e \int_0^t \|\rho_s\|_{L^\infty}^2\dd s.
\end{equation}
\end{lemma}

\begin{proof}
First let us discuss about some properties of $H.$ It is smooth, its gradient is $\ds\nabla H(z)=\frac{z}{H(z)}$ and it satisfies $\nabla H(z)\cdot z\le H(z),\ \forall z\in\R^d.$ Moreover its Hessian matrix is given by
$$[H_{ij}(z)]_{i,j\in\{1,\dots,d\}}=\left[\frac{\d^{ij}H^2(z)-z^iz^j}{H^3(z)}\right]_{i,j\in\{1,\dots,d\}}=\frac{1}{H(z)}I_d-\frac{1}{H^3(z)}z\otimes z,\ \forall\ z\in\R^d,$$ 
where 
$\ds
\d^{ij}=\left\{
\begin{array}{ll}
1, & {\rm{if}}\ i=j,\\
0, & {\rm{if}}\ j\neq j,
\end{array}
\right.
$
is the Kronecker symbol. Note that, from this computation, the matrix $D^2H\geq 0$ is bounded from above by $\ds\frac 1H$, and hence by $\e^{-1}$. Moreover we introduce a uniform constant $C>0$ such that $\|v\|^2_{L^{\infty}}|\Omega|+\|\nabla\cdot v\|_{L^{\infty}}+\|Dv\|_{L^\infty}\le C.$

Now to show the estimate of this lemma we calculate the quantity $\ds\frac{\dd}{\dd t}\int H(\nabla\rho_t)\dd x.$
\begin{align*}
\frac{\dd}{\dd t}\int_\Om H(\nabla\rho_t)\dd x & = \int_\Om \nabla H(\nabla\rho_t)\cdot\partial_t\nabla\rho_t\dd x=\int_\Omega\nabla H(\nabla\rho_t)\cdot\nabla (\Delta\rho_t-\nabla\cdot(v_t\rho_t))\dd x\\
& = \int_\Omega\nabla H(\nabla\rho_t)\cdot\nabla\Delta\rho_t\dd x-\int_\Omega\nabla H(\nabla\rho_t)\cdot\nabla (\nabla\cdot (v_t\rho_t))\dd x\\
&=: (I) + (II)
\end{align*}
Now we study each term separately and for the simplicity we drop the $t$ subscripts in the followings.  We start from the case of the torus, where there is no boundary term in the integration by parts.
$$
\begin{array}{rcl}
\vspace{5pt}
(I) &=&\ds\int_\Omega\nabla H(\nabla\rho)\cdot\nabla\Delta\rho\dd x=\int_\Omega \sum_{j,i} H_j(\nabla\rho)\rho_{jii}\dd x=-\int_\Omega\sum_{j,i,k}H_{kj}(\nabla\rho)\rho_{ik}\rho_{ji}\dd x \vspace{20pt}\\

(II) & =&\ds -\int_\Omega\nabla H(\nabla\rho)\cdot\nabla (\nabla\cdot (v\rho))\dd x=-\int_\Omega\sum_{i,j}H_j(\nabla\rho)(v^i\rho)_{ij}\dd x\\ \vspace{5pt}
\ &= & \ds \int_\Omega\sum_{i,j,k}H_{jk}(\nabla\rho)\rho_{ki} v^i_j\rho\dd x +\int_\Omega\sum_{i,j,k}H_{jk}(\nabla\rho)\rho_{ki} v^i\rho_j\dd x\\ \vspace{5pt}
\ &=: & (II_a) + (II_b).
\end{array}
$$

First look at the term $(II_a)$. Since the matrix $H_{jk}$ is positive definite, we can apply a Young inequality for each index $i$ and obtain
\begin{eqnarray*}
(II_a)=\int_\Omega\sum_{i,j,k}H_{jk}(\nabla\rho)\rho_{ki} v^i_j\rho\dd x&\leq& \frac 12 \int_\Omega\sum_{i,j,k}H_{jk}(\nabla\rho)\rho_{ki} \rho_{ij}\dd x+\frac 12 \int_\Omega\sum_{i,j,k}H_{jk}(\nabla\rho)v^i_j v^i_k\rho^2\dd x\\&\leq& \frac 12 |(I)|+C\|\rho\|_{L^2}^2\|D^2H\|_{L^\infty}.
\end{eqnarray*}
The $L^2$ norm in the second term will be estimated by the $L^\infty$ norm for the sake of simplicity (see Remark \ref{rem torus L^2} below).

For the term $(II_b)$ we first make a point-wise computation
\begin{align*}
\sum_{i,j,k}H_{jk}(\nabla\rho)\rho_{ki} v^i\rho_j&=\frac{1}{H^3(\nabla\rho)}\sum_i[D^2_i\rho\cdot\left(\e^2I_d+|\nabla\rho|^2 I_d-\nabla\rho\otimes\nabla\rho\right)\cdot\nabla\rho]v^i\\
&=\frac{\e^2}{H^3(\nabla\rho)}\sum_iv^iD^2_i\rho\cdot\nabla\rho=-\e^2 \sum_i v^i \partial_i\left(\frac{1}{H(\nabla\rho)}\right).
\end{align*}
where $D^2_i\rho$ denotes the $i^{th}$ row in the Hessian matrix of $\rho$ and we used $\left(|\nabla\rho|^2 I_d-\nabla\rho\otimes\nabla\rho\right)\cdot\nabla\rho=0.$

Integrating by parts we obtain
$$(II_b)= \e^2\int_\Omega (\nabla\cdot v)\frac{1}{H(\nabla\rho)}\dd x\leq C\e^2\|1/H\|_{L^\infty}\leq C\e,$$
where we used $H(z)\geq \e$.

Summing up all the terms we get and using $\|D^2H\|\leq \e^{-1}$  we get
$$\frac{\dd}{\dd t}\int_\Om H(\nabla\rho_t)\dd x \leq -\frac12 |(I)|+C\|\rho_t\|_{L^\infty}^2\|D^2H\|_{L^\infty}+C\e\leq C\e+C\|\rho_t\|_{L^\infty}^2\e^{-1},$$
which proves the claim.

If we switch to the case of a smooth bounded convex domain $\Omega$, we have to handle boundary terms. These terms are
$$\int_{\partial\Omega} \sum_{i,j}H_j(\nabla\rho)\rho_{ij} n^i-\int_{\partial\Omega} \sum_{i,j}H_{j}(\nabla\rho)\rho v^i_{j} n^i,$$
where we ignored those terms involving $n^i v^i$ (i.e., the integration by parts in $(II_b)$, and the term $H_j(\nabla\rho)\rho_j n^i v^i$ in the integration by parts of $(II_a)$), since we already supposed $v\cdot n=0$. We use here Lemma \ref{lemma exball}, which provides 
$$\sum_{i,j}H_j(\nabla\rho)\rho_{ij} n^i-\rho H_j(\nabla\rho)v^i_j n^i=\frac{1}{H(\nabla\rho)}\sum_{i,j}\left(\rho_j\rho_{ij} n^i -\rho \rho_j v^i_j n^i\right)=-\frac{1}{H(\nabla\rho)}\sum_{i,j}\left(\rho_j h_{ij} \rho_i -\rho \rho_j h_{ij} v^i\right).$$
If we use the fact that the matrix $D^2h $ is positive definite and a Young inequality, we get $\sum_{i,j}\rho_j h_{ij} \rho_i \geq 0 $ and 
$$\rho\sum_{i,j}| \rho_j h_{ij} v^i|\leq \frac 12 \sum_{i,j}\rho_j h_{ij} \rho_i+\frac 12\sum_{i,j} \rho^2 v^j h_{ij} v^i,$$
which implies 
$$\frac{1}{H(\nabla\rho)}\sum_{i,j}\left(\rho_j\rho_{ij} n^i -\rho \rho_j v^i_j n^i\right)\leq \frac{\rho^2}{H(\nabla\rho)}\|D^2h\|_{L^\infty} |v|^2\leq \frac{C\|\rho\|_{L^\infty}^2}{\e}.$$
This provides the desired estimate on the boundary term.
\end{proof}

\begin{remark}\label{rem torus L^2} 
In the above proof, we needed to use the $L^\infty$ norm of $\rho$ only in the boundary term. When there is no boundary term, the $L^2$ norm is enough, in order to handle the term $(II_a)$. In both cases, the norm of $\rho$ can be bounded in terms of the initial norm multiplied by $e^{Ct}$, where $C$ bounds the divergence of $v$. On the other hand, in the torus case, one only needs to suppose $\rho_0\in L^2$ and in the convex case $\rho_0\in L^\infty$. Both assumptions are satisfied in the applications to crowd motion with density constraints.
\end{remark}

We have seen that the constants in the above inequality depend on $\e$ and explode as $\e\to 0$. This prevents us to obtain a clean estimate on the $BV$ norm in this context, but at least proves that $\rho_0\in BV\Rightarrow \rho_t\in BV$ for all $t>0$ (to achieve this result, we just need to take $\e=1$). Unfortunately, the quantity which is estimated is not the $BV$ norm, but the integral $\ds\int_\O H(\nabla\rho)$. This is not enough for the purpose of the applications to Section \ref{sec:bv}, as it is unfortunately not true that the projection operator decreases the value of this other functional\footnote{Here is a simple counter-example: consider $\mu=g(x)\dd x$ a $BV$ density on $[0,2]\subset\R$, with $g$ defined as follows. Divide the interval $[0,2]$ into $2K$ intervals $J_i$ of length $2r$ (with $2rK=1$); call $t_i$ the center of each interval $J_i$ (i.e. $t_i=i2r+r$, for $i=0,\dots,2K-1$) and set $g(x)=L+\sqrt{r^2-(x-t_i)^2}$ on each $J_i$ with $i$ odd, and $g(x)=0$ on $J_i$ for $i$ even, taking $L=1-\pi r/4$. It is not difficult to check that the projection of $\mu$ is equal to the indicator function of the union of all the intervals $J_i$ with $i$ odd, and that the value of $\int H(\nabla\rho)$ has increased by $K(2-\pi/2)r=1-\pi/4$, i.e. by a positive constant (see Figure \ref{doigts}).
}.

\begin{figure}[h]
\begin{minipage}{4cm}
\begin{tikzpicture}[scale=0.6,xmin=-0.6,xmax=4.8,ymin=-0.5,ymax=5]
\axes \fenetre
\draw (0.2,0)--(0.2,4) arc(180:0:0.3)--(0.8,0)--(1.4,0)--(1.4,4) arc(180:0:0.3)--(2,0)--(2.6,0)--(2.6,4) arc(180:0:0.3)--(3.2,0)--(3.8,0)--(3.8,4) arc(180:0:0.3)--(4.4,0);
\draw[dashed] (0,4)--(4.5,4);
\draw[dashed] (0,4.2)--(4.5,4.2);
\draw(-0.4,3.8) node{$L$};
\draw (-0.4,4.4) node{$1$};
\draw (4.6,4.6) node{$\mu$};
\end{tikzpicture}
\end{minipage}
\begin{minipage}{4.5cm}
\begin{tikzpicture}[scale=0.6,xmin=-0.6,xmax=5.3,ymin=-0.5,ymax=5]
\axes \fenetre
\draw (0.2,0)--(0.2,4.2) --(0.8,4.2)--(0.8,0)--(1.4,0)--(1.4,4.2) -- (2,4.2)--(2,0)--(2.6,0)--(2.6,4.2) -- (3.2,4.2)--(3.2,0)--(3.8,0)--(3.8,4.2)--(4.4,4.2)--(4.4,0);
\draw[dashed] (0,4)--(4.5,4);
\draw[dashed] (0,4.2)--(4.5,4.2);
\draw(-0.4,3.8) node{$L$};
\draw (-0.4,4.4) node{$1$};
\draw (4.6,4.6) node{$P_\cK[\mu]$};
\end{tikzpicture}
\end{minipage}
\caption{The counter-example to the decay of $\ds\int_\O H(\nabla\rho)$, which corresponds to the total legth of the graph}
\label{doigts}
\end{figure}
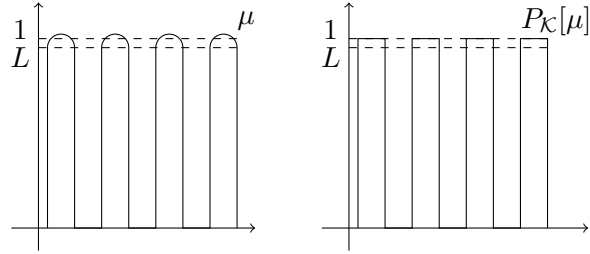

If we stay interested to the value of the $BV$ norm, we can provide the following estimate.

\begin{lemma}\label{bv_estimate sqrte}
Under the assumptions of Lemma \ref{bv_estimate2}, if we suppose $\rho_0\in BV(\Omega)\cap L^\infty(\Omega)$, then, for $t\leq T$, we have\begin{equation}
\int_\Omega |\nabla\rho_t|\dd x\le \int_\Omega |\nabla\rho_0|\dd x  +C\sqrt{t},
\end{equation}
where the constant $C$ depends on $v$, on $T$ and on $\|\rho_0\|_{L^\infty}$.
\end{lemma}

\begin{proof}
Using the $L^\infty$ estimate of Lemma \ref{tool3}, we will assume that $\|\rho_t\|_{L^\infty}$ is bounded by a constant (which depends on 
$v$, on $T$ and on $\|\rho_0\|_{L^\infty}$). Then, we can write
$$\int_\Omega |\nabla\rho_t|\dd x\le \int_\Omega H(\nabla\rho_t)\dd x\leq  \int_\Omega H(\nabla\rho_0)\dd x+C\e t+\frac{Ct}{\e}\leq 
 \int_\Omega (|\nabla\rho_0|+\e)\dd x+C\e t+\frac{Ct}{\e}.$$
 It is sufficient to choose, for fixed $t$, $\e=\sqrt{t}$, in order to prove the claim.
\end{proof}

Unfortunately, this $\sqrt{t}$ behavior is not suitable to be iterated, and the above estimate is useless for the sake of Section \ref{sec:bv}.
The existence of an estimate (for $v$ Lipschitz) of the form $TV(\rho_t)\le TV(\rho_0)+Ct$, or $TV(\rho_t)\le TV(\rho_0)e^{Ct}$, or even $f(TV(\rho_t))\leq f(TV(\rho_0))e^{Ct}$, for any increasing function $f:R_+\to\R_+$, seems to be an open question. 
\bigskip

{\sc{Acknowledgements.}} The authors warmly acknowledge the support of the ANR project ISOTACE (ANR-12-MONU-0013) and of the iCODE project ``strategic crowds'', funded by the IDEX {\it Universit\'e Paris-Saclay}. They also acknowledge the warm hospitality of the {\it Fields Institute} of Toronto, where a large part of the work was accomplished during the Thematic Program on Variational Problems in Physics, Economics and Geometry in Fall 2014. Last but not least, the two anonymous referees who read very carefully the paper and suggested many useful improvements are warmly thanked as well.

\bibliographystyle{alpha}

\end{document}